\documentclass[12pt]{article}
\usepackage{amssymb,amsmath,latexsym}
\usepackage{amsthm}
\usepackage{enumitem}

\usepackage{caption}
\usepackage{xcolor}

\newcommand{\R}{\mathbb R}

\newcommand{\F}{\mathbb F}
\renewcommand{\P}{\mathbb P}
\renewcommand{\S}{\mathbb S}

\DeclareMathOperator{\GL}{GL}

\DeclareMathOperator{\Sym}{Sym}

\DeclareMathOperator{\Aut}{Aut}
\DeclareMathOperator{\Gal}{Gal}

\DeclareMathOperator{\Tr}{Tr}
\DeclareMathOperator{\diag}{diag}

\newtheorem{theorem}{Theorem}[section]
\newtheorem{notation}[theorem]{Notation}
\newtheorem{definition}[theorem]{Definition}
\newtheorem{lemma}[theorem]{Lemma}
\newtheorem{corollary}[theorem]{Corollary}
\newtheorem{proposition}[theorem]{Proposition}
\newtheorem{remark}[theorem]{Remark}
\newtheorem{construction}[theorem]{Construction}

\newtheorem{problem}[theorem]{Open Problem}

\DeclareMathOperator{\id}{id}
\DeclareMathOperator{\End}{End}
\usepackage[hyphens]{url}
\usepackage{breakurl}
\usepackage[colorlinks=true,citecolor=blue,linkcolor=blue,urlcolor=blue,bookmarks,bookmarksopen,bookmarksdepth=2,backref=page,breaklinks]{hyperref}

\usepackage{bookmark}
\bookmarksetup{
	numbered, 
	open,
}

\renewcommand*{\backref}[1]{}
\renewcommand*{\backrefalt}[4]{%
	\ifcase #1 (Not cited.)%
	\or        p.~#2.%
	\else      pp.~#2.%
	\fi}

\begin{document}
\title{A unifying construction of semifields of order $p^{2m}$}
\author{Lukas K\"olsch}
\date{\today}
\maketitle
\abstract{In this article, we present two new constructions for semifields of order $p^{2m}$. Together, the constructions unify and generalize  around a dozen distinct semifield constructions, including both the oldest known construction by Dickson and the largest known constructions in odd characteristic by Taniguchi. The constructions also provably yield many new semifields. We give precise conditions when the new semifields we find are equivalent and count precisely how many new inequivalent semifields we construct.  }
\section{Introduction}
\subsection{Overview}
A \emph{finite semifield} $\S$ is a finite division algebra in which multiplication is not assumed to be associative. Both left- and right-multiplication with a fixed element of $\S$ define an invertible linear mapping, i.e., if $\circ$ denotes the semifield multiplication then the equations $a \circ x = b$ and $x \circ a = b$ have a unique solution $x \in \S$ for any $a,b \in \S$. The first non-trivial example of a finite semifield (i.e. a semifield that is not a field) was found by Dickson in 1906~\cite{dickson1906commutative}. 

Finite semifields have received much attention since, in particular in finite geometry where semifields can be used to define projective planes, early constructions and connections were for instance developed by Knuth~\cite{Knuth65} and Albert~\cite{Albert60} in the 1960s, we also refer to some more recent results establishing further connections~\cite{blokhuis2003classification,KW}. Recently, it has been observed that semifields can be used to construct linear rank-metric codes with optimal parameters, see e.g.~\cite{sheekey} for details. Further connections to combinatorial objects like difference sets and bent functions have also been noted, we refer to the survey~\cite{pott2014semifields}.

Since Dickson's original construction, many different constructions have been proposed using a variety of tools, ranging from planar functions~\cite{golouglu2023exponential,ZKW}, skew-polynomial rings~\cite{lavrauw2013semifields,sheekey2020new}, to  spreads~\cite{Kantor03}, to name a few. Naturally, it is quite difficult to determine if these constructions are actually distinct or lead to \emph{equivalent} objects--- the precise definition of equivalence we use will be described later. Curiously, many ostensibly very different constructions only yield semifields in even dimension over the prime field, i.e. semifields of size $p^{2m}$. \\

In this paper, we will develop two general construction methods that yield semifields of size $p^{2m}$ for any prime $p$ and any natural number $m$. The contribution of these new constructions to the theory of semifields is two-fold: Firstly, the constructions provide a general and unifying framework that explains most known infinite families of semifields of order $p^{2m}$. This is in particular surprising as the constructions were originally found using very different ideas and tools, often using ad-hoc arguments and extrapolations from brute force searches, aided by computers. Our findings show that all these families of semifields are actually closely related. In this way, we cover the bivariate (twisted) cyclic fields~\cite{jha1989analog,sheekey2020new}, the (generalized) Dickson semifields~\cite{dickson1906commutative,Knuth65}, Knuth's semifields quadratic over a weak nucleus~\cite{Knuth65} (including the Hughes-Kleinfeld semifields \cite{hughes1960seminuclear}), Bierbrauer's semifields in odd~\cite{Bierbrauer16} and even~\cite{bartoli2017family} characteristic (containing in particular the Budaghyan-Helleseth commutative semifields), Dempwolff's semifields~\cite{dempwolff2013more}, Zhou-Pott's commutative semifields~\cite{ZP13}, and Taniguchi's semifields~\cite{taniguchisf}. In all cases, we give new proofs using our framework that span only a few lines; dealing with even and odd characteristic simultaneously. We want to emphasize that in this way, our framework covers the oldest known semifields (Dickson's semifields) as well as the currently largest known family of semifields (Taniguchi's semifields~\cite{golouglu2023counting}).

Secondly, we present and analyze a new family of semifields that our framework provides. This family intersects with Zhou-Pott's family~\cite{ZP13}, but we prove that it contains many new examples as well that are not contained in any other infinite family of semifields. 

\subsection{Structure of the paper}

We start in Section~\ref{sec:prelim} by giving the necessary definitions and previous results. We also mention some connections to finite geometry. In particular, we define when we call two semifields equivalent which will be instrumental in later sections. We also examine the twisted cyclic semifields which is a family of semifields found in 2020 by Sheekey~\cite{sheekey2020new}. As we will explain, the twisted cyclic semifields of size $p^{2m}$ can be seen as a "trivial" case of our constructions (although twisted cyclic semifields of other sizes are not covered by our constructions). 

In Section~\ref{sec:newconstructions}, we present Construction~\ref{thm:general_structure} and Construction~\ref{thm:general_structure_2} which are the main contribution of the paper. These constructions need as "input" a (what we call)  \emph{admissible mapping}, and Section~\ref{sec:admissible} is devoted to find these admissible mappings. Next to a trivial admissible mapping (which returns the twisted cyclic semifields), we provide two non-trivial admissible mappings (Propositions~\ref{prop:admissible1} and~\ref{prop:admissible2}). Finding these admissible mappings requires a classification result on certain irreducible semilinear transformations over a finite field which we provide in Proposition~\ref{prop:irred}.

In Section~\ref{sec:known}, we present all known infinite families of semifields that can be constructed by our unified construction technique. As pointed out in the introduction, this covers (depending on how one counts the families) around 10 distinct families, see Table~\ref{t:1} for a brief overview showing which family of semifields can be recovered by which construction.

In Section~\ref{sec:new}, we investigate new semifields (i.e., semifields not covered by any known construction) that are produced by our framework. We show that this new family can be seen as a (vast) generalization of the Zhou-Pott semifields found in 2013~\cite{ZP13}. Specifically, we prove that the new family contains many semifields that are inequivalent to the Zhou-Pott semifields. Indeed, Theorem~\ref{thm:isotop} and Corollary~\ref{cor:ZP} show that the number of inequivalent semifields of order $p^{2m}$ in this family is exponential in $m$, while the Zhou-Pott semifield family only yields a quadratic number of inequivalent examples. We also prove exactly when two semifields in the new family are equivalent to each other, as well as prove some other properties of these new semifields. Lastly, we prove that our new constructions contain semifields that are not contained in \emph{any} other construction of semifields. 

The last section concludes the paper with a list of open questions, in particular we discuss some connections to coding theory and finite geometry.

\section{Preliminaries} \label{sec:prelim}
A \textbf{semifield} $\S = (S,+,\circ)$ is a set $S$ equipped with two operations $(+,\circ)$
satisfying the following axioms. 
\begin{enumerate}
\item[(S1)] $(S,+)$ is a group.
\item[(S2)] For all $x,y,z \in S$,
\begin{itemize}
\item $x\circ (y+z) = x \circ y + x \circ z$,
\item $(x+y)\circ z = x \circ z + y \circ z$.
\end{itemize}
\item[(S3)] For all $x,y \in S$, $x \circ y = 0$ implies $x=0$ or $y=0$.
\item[(S4)] There exists $\epsilon \in S$ such that $x\circ \epsilon = x = \epsilon \circ x$.
\end{enumerate}

In this paper, we will be interested only in finite semifields, and we assume all semifields to be finite. 
An algebraic object satisfying the first three of the above axioms is called 
a \textbf{pre-semifield}. The additive group of a (pre)-semifield is 
always an elementary abelian $p$-group \cite[p. 185]{Knuth65}, and can thus be viewed as
an $n$-dimensional $\F_{p}$-vector space. 
If $\circ$ is associative then $\S$ is necessarily a finite field by 
Wedderburn's theorem. 

A pre-semifield $(\F_p^n,+,\circ)$ 
can be converted to a semifield $\S = (\F_p^n,+,\ast)$ using {\em Kaplansky's trick} 
by defining the new multiplication as
\[
	(x \circ e) \ast (e \circ y) = (x \circ y),
\]
for any nonzero element $e \in \F_p^n$, making $(e \circ e)$ the multiplicative 
identity of $\S$. A pre-semifield is an $\F_{p}$-algebra, thus the multiplication
is bilinear, i.e. we have $\F_{p}$-linear left and right multiplications $L_x,R_y : \F_p^n \to \F_p^n$, with
\[
L_x(y) := x \circ y =: R_y(x).
\]
The mappings $L_x$ and $R_y$ are bijections whenever $x \ne 0$ (resp. $y \ne 0$) 
by (S3). 

Two pre-semifields $\P_1 = (\F_p^n,+,\circ_1)$ and $\P_2 = (\F_p^n,+,\circ_2)$ 
are said to be \textbf{isotopic} if there exist $\F_{p}$-linear bijections 
$N_1,N_2$ and $N_3$ of $\F_p^n$ satisfying
\[
N_1(x \circ_1 y) = N_2(x) \circ_2 N_3(y).
\]
Such a triple $\gamma = (N_1,N_2,N_3) \in \GL(\F_p^n)^3$ is called an \textbf{isotopism} between $\P_1$ and $\P_2$. 

Isotopisms between a pre-semifield $\P$ and itself are called \textbf{autotopisms} and the set of all autotopisms forms the \textbf{autotopism group} $\Aut(\P)$. In particular, the pre-semifield $\P$ and the corresponding semifield $\S$ constructed 
by Kaplansky's trick are isotopic. 

As we will outline in the next subsection, semifields can be used to construct projective planes and isotopic semifields yield isomorphic planes~\cite{Albert60}. We will thus consider isotopic semifields equivalent. 

Associative substructures of a semifield $\S = (\F_p^n,+,\ast)$, 
namely the \textbf{left, middle and right nuclei}, are defined as follows:
\begin{align*}
\mathcal{N}_l(\S) &:= \{ x \in \F_p^n \ : \ (x \ast y) \ast z = x \ast (y \ast z), \ y,z \in \F_p^n \},\\
\mathcal{N}_m(\S) &:= \{ y \in \F_p^n \ : \ (x \ast y) \ast z = x \ast (y \ast z), \ x,z \in \F_p^n \},\\
\mathcal{N}_r(\S) &:= \{ z \in \F_p^n \ : \ (x \ast y) \ast z = x \ast (y \ast z), \ x,y \in \F_p^n \}.
\end{align*}
It is easy to check that $\mathcal{N}_l(\S),\mathcal{N}_m(\S),\mathcal{N}_r(\S)$ are 
finite fields, and it is elementary to verify that $\S$ can then be viewed as a (left or right) vector space over its nuclei. This fact has been used in many previous constructions, and these constructions then necessarily lead to semifields with at least one large nucleus (see e.g.~\cite{CARDINALI2006940,ebert2009infinite,Knuth65}). As we will see, our constructions eschew this approach -- indeed we show that the new constructions produce both semifields with small and large nuclei, depending on the choice of parameters.

Nuclei are isotopy invariants for semifields and (since every pre-semifield is isotopic to a semifield via Kaplansky's trick) we can extend the definition of nuclei to pre-semifields as well. Based on this isotopy via Kaplansky's trick, we will not distinguish between pre-semifields and semifields from now on and just use the term \emph{semifield} for both objects.

\subsection{Some geometric connections and the Knuth orbit}
In this subsection we briefly outline the basic connections between semifields and finite geometry. For comprehensive treatments, we refer the reader to~\cite{dembowski1997finite,hughesbook}. The projective plane $\Pi(\S)$ is constructed from a semifield $\S$ of order $q$ as follows: 

Denote the $q^2+q+1$ points by $(0,0,1),(0,1,a),(a,b,1)$ for $a,b \in \S$, and similarly the lines by $[0,0,1],[0,1,a],[a,b,1]$. Then a point $(x,y,z)$ is incident with the line $[u,v,w]$ if $uz=y \circ v+xw$, where we define $1s=s1=s$ and $0s=s0=0$ for all $a \in \S$. It is easy to check that two points of $\Pi(\S)$ are incident with exactly one line and any two lines are incident with exactly one point, so $\Pi(\S)$ is indeed a projective plane. We say such a $\Pi$ is \emph{coordinatized} by $\S$. If the semifield is not a field, this semifield plane will be a non-desarguesian translation plane. 

A semifield $\S=(\F_p^n,+,\ast)$ can be used to construct a \textbf{spread} $\Sigma$, i.e., a partition of the non-zero elements of $\F_p^n \times \F_p^n$ into disjoint subspaces with $p^n$ elements:
\begin{equation}
		\Sigma = \{\{(y,y\circ x)\colon y \in \F_p^n\} \colon x \in \F_{p}^n\} \cup \{(0,y) \colon y \in \F_p^n\} .
\label{eq:ABB}
\end{equation}
This semifield spread itself can be used to construct the projective plane $\Pi$ directly via the Andr\'e-Bruck-Bose construction, see~\cite{dembowski1997finite}.

A major result (which motivates the definition of isotopy of semifields) states that two semifield planes are isomorphic if and only if the corresponding semifields are isotopic~\cite{Albert60}.

Outside of isotopy, there is a further equivalence on semifields that we need to examine, namely the \textbf{Knuth orbit}, which was defined by Knuth in 1965~\cite{Knuth65} using planar ternary rings and tensors. We follow here a more geometric view introduced in~\cite{ball2004six}. 

The first equivalence is very simple: If one has a semifield $\S= (\F_p^n,+,\circ)$, one defines the \textbf{dual semifield} $\pi_1(\S)=(\F_p^n,+,\ast)$ by defining $x \ast y := y \circ x$ for all $x,y \in \S$. The plane $\Pi(\pi_1(\S))$ is called the \textbf{dual plane} of $\S$, and semifield planes can be characterized as precisely the translations planes whose dual is also a translation plane~\cite[Chapter 8]{hughesbook}.

A way to get another semifield is to construct the dual spread of the spread defined by $\S$ in Eq.~\eqref{eq:ABB}. The dual spread consists of all spaces dual to the spaces in Eq.~\eqref{eq:ABB}. This dual spread again coordinatizes a semifield plane. We call the corresponding semifield the \textbf{transposed semifield} of $\S$ and denote it by $\pi_2(\S)$. We will give an example of how to compute the dual spread (and thus the transposed semifield) explicitly in Section~\ref{sec:known}.

Crucially, both $\pi_1(\S)$ and $\pi_2(\S)$ are generally \emph{not} isotopic to $\S$. The operations $\pi_1$ and $\pi_2$ generate a group isomorphic to $\Sym(3)$, and in total one semifield can thus yield six non-isotopic semifields, namely $\S$, $\pi_1(\S)$, $\pi_2(\S)$, $\pi_2\pi_1(\S)$, $\pi_1\pi_2(\S)$, $\pi_1\pi_2\pi_1(\S)$. These six semifields are called the \textbf{Knuth orbit} of $\S$. The $\Sym(3)$-action of the Knuth orbit extends to the three nuclei of a semifield. Specifically, $\pi_1$ exchanges left and right nucleus of a semifield and $\pi_2$ exchanges right and middle nucleus (up to isomorphy), see e.g.~\cite[Proposition 2.4.]{MP12}. 

In Section~\ref{sec:known}, the Knuth orbit will play a crucial rule. It turns out that many known semifields (see Table~\ref{t:1}) are actually \emph{not} isotopic to one of our constructions, however they are isotopic to a semifield that is in the Knuth orbit of our construction --- and thus still equivalent to it. 
\begin{definition}
	We call two semifields $\S_1$ and $\S_2$ of the same order equivalent if $\S_1$ is isotopic to a semifield in the Knuth orbit of $\S_2$. 
\end{definition}
In particular, the orders of the three nuclei together form an equivalence invariant for semifields.

\subsection{Notation}

We list notation that we will use throughout the paper. Sometimes we will opt to restate some of the notations to emphasize a point or to make a result self-contained.
\begin{notation}\label{not}
	\begin{itemize}
	\item []
		\item $L=\F_{p^m}$ is a finite field,
		\item $\sigma \colon x \mapsto x^{p^k}$, $\tau \colon x \mapsto x^{p^l}$ are field automorphisms of $L$,
		\item $\overline{\sigma}$, $\overline{\tau}$ are the inverses of $\sigma$ and $\tau$ in $\Gal(L)$,
		\item $K=\F_{p^{\gcd(k,m)}}$ is the fixed field of $\sigma \in \Gal(L)$, 
		\item $N_{L\colon K}(x) = x^{\frac{p^m-1}{p^{\gcd(k,m)}-1}}$ is the norm mapping of $L$ into $K$, 
		\item $\mathbf{x}$, $\mathbf{y}$ are vectors in $L^d$, for most of the paper with $d=2$,
		\item $\mathbf{x}^{\sigma}$ denotes the vector that results after applying $\sigma$ to $\mathbf{x}$ component-wise.
		\item $T \in \Gamma L (d,L)$ is a semilinear operator with associated field automorphism $\sigma$, again for most of the paper we will consider only $d=2$,
		\item With slight abuse of notation, we will denote for instance the $(p^k+1)$-st powers as $(\sigma+1)$-st powers. Similarly, we write for instance $x^{\sigma+1}$ for $x^{p^k+1}$.
	\end{itemize}
\end{notation}
We want to briefly note that some results stated in this paper for finite fields $L$ with automorphism $\sigma$ also hold (sometimes with minor tweaks) for arbitrary fields with cyclic Galois group over the fixed field $K$. 
\subsection{Twisted cyclic semifields}

One of the most classic and general semifield constructions was found in 1989 by Jha and Johnson~\cite{jha1989analog}, called the \emph{cyclic semifields.} 

\begin{definition}
	An element $T \in \Gamma L(d,L)$ is called irreducible if the only invariant subspaces of $T$ are $\{0\}$ and $L^d$. 
\end{definition}

\begin{theorem}\cite{jha1989analog}
	Let $T$ be an irreducible element of $\Gamma L(d,L)$. Fix an $L$-basis of $ V=L^d$, say $\{e_0,\dots,e_{d-1}\}$. Define a multiplication 
	\[\mathbf{x} \circ \mathbf{y} = \sum_{i=0}^{d-1}y_iT^i(\mathbf{x}),\]
	where $\mathbf{y}=\sum_{i=0}^{d-1}y_ie_i$. Then $(V,+,\circ)$ is a semifield, called a cyclic semifield.
\end{theorem}

This construction turns out to be equivalent to one using skew polynomial rings, see~\cite{lavrauw2013semifields}. 

The cyclic semifields were generalized in 2020 by Sheekey~\cite{sheekey2020new}. His construction, too, uses skew-polynomials. We eschew skew polynomials and formulate his results using irreducible semilinear transformations, for the translation we refer to~\cite{lavrauw2013semifields}. Changing the "language" from skew-polynomials to semilinear transformations allows us to use concepts from matrix theory that substantially simplify our approach.

Let us first introduce a convention we will use for semilinear transformations from now on. We will write $T \in \Gamma L(d,L)$ with associated field automorphism $\sigma$ as
\[T\begin{pmatrix}
	x_1\\
	\vdots\\
	x_d
\end{pmatrix} = M \mathbf{x}^\sigma = M\begin{pmatrix}
	x_1^\sigma\\
	\vdots\\
	x_d^\sigma
\end{pmatrix},\]
where $M \in \GL(d,L)$. We denote the linear mapping (or, after fixing an arbitrary basis, matrix) $M$ that belongs to $T$ by $M_T$. We are now ready to state the definition of the twisted cyclic semifields {\cite[Theorem 6]{sheekey2020new}}.

\begin{theorem}[Twisted cyclic semifields]\label{thm:sheekey}
	Let $T$ be an irreducible element in $\Gamma L(d,L)$ with associated field automorphism $\sigma$ of order $t$ and fixed field $K$. Let further $\rho$ be an automorphism of $L$ with fixed field $K'\leq K$ and $\eta \in L$ chosen such that
		\[N_{L\colon K'}(\eta)N_{K\colon K'}((-1)^{d(t-1)}\det(M_T))\neq 1.\]
	
	Define a binary operation on $V$ via
	\[\mathbf{x} \circ \mathbf{y} = \sum_{i=0}^{d-1}y_iT^i(\mathbf{x})+\eta y_0 T^d(\mathbf{x}),\]
	where $\mathbf{y}=\left(\begin{smallmatrix}y_0\\\vdots\\y_{d-1}\end{smallmatrix}\right)$. Then $\S_T=(V,+,\circ)$ is a semifield, called a twisted cyclic semifield.
\end{theorem}
Note that $\eta=0$ recovers the cyclic semifields. 

The only semifield axiom that is not obvious to prove is that the right-multiplications $R_\mathbf{y}(\mathbf{x})=\mathbf{x}\circ \mathbf{y}$ for $\mathbf{y} \neq 0$ are all non-singular. The main ingredient of the proof of Theorem~\ref{thm:sheekey} is thus the following, again translated from the skew-polynomial setting to our notation:
\begin{theorem}{\cite[Theorem 5]{sheekey2020new}}\label{thm:sheekeysingular}
	Let $\mathbf{x},\mathbf{y}\in L^d$ and $T$ be an irreducible transformation in $\Gamma L(d,L)$ with associated field automorphism $\sigma$ of order $t$ with fixed field $K$. Then the mappings $F_{\mathbf{y}} \colon L^d \rightarrow L^d$ defined by
	\[F_{\mathbf{y}}(\mathbf{x})=\sum_{i=0}^{d-1}y_iT^i(\mathbf{x})+ y_dT^d(\mathbf{x})\]
	are non-singular for any $0\neq \mathbf{y}=\left(\begin{smallmatrix}y_0\\\vdots\\y_{d-1}\end{smallmatrix}\right)$ if and only if $y_d=0$ or  $$N_{L\colon K}(y_0/y_d)\neq (-1)^{d(t-1)}N_{L\colon K}(\det(M_T))$$ if $y_d\neq 0$.
\end{theorem}

In the following section we will make frequent use of this result in the following form:
\begin{corollary} \label{cor:sheekey}
	Let $\mathbf{x}\in L^d$ and $T$ be an irreducible transformation in $\Gamma L(d,L)$ with associated field automorphism $\sigma$ of order $t$ with fixed field $K$. Then the mappings $F_{y_1,\dots,y_{d-1}} \colon L^d \rightarrow L^d$ defined by
	\[F_{y_1,\dots,y_{d-1}}(\mathbf{x})=\sum_{i=1}^{d-1}y_iT^{i-1}(\mathbf{x})+\eta T^{d-1}(\mathbf{x})+ \det(M_T)^{\overline{\sigma}} T^{-1}(\mathbf{x}) \]
	for any $y_1,\dots,y_{d-1}$ are non-singular for any $\eta \in L$ with $N_{L\colon K}(\eta)\neq (-1)^{d(t-1)}$.
\end{corollary}
\begin{proof}
We use Theorem~\ref{thm:sheekeysingular}. We apply a transformation $\mathbf{x}\mapsto T^{-1}\mathbf{x}$ and then set $y_0=\det(M_T)^{\overline{\sigma}} \neq 0$ and $y_d=\eta$. Then $N_{L\colon K}(y_0)=N_{L\colon K}(\det(M_T))$, so the condition in Theorem~\ref{thm:sheekeysingular} boils down to $N_{L\colon K}(\eta)\neq (-1)^{d(t-1)}$.
\end{proof}


In the next section, we will give constructions of semifields where the right-multiplications will be non-singular mappings of the form in Corollary~\ref{cor:sheekey} for the special case $d=2$. It turns out some of these are known already (although our framework will show that they are all related), but we also find many new examples; in particular we will construct many examples which are \emph{not} twisted cyclic semifields.

\section{Two general constructions of semifields of order $p^{2m}$} \label{sec:newconstructions}
We will now concentrate on constructions of semifields of order $p^{2m}$. In particular, we will be using Theorem~\ref{thm:sheekeysingular} and Corollary~\ref{cor:sheekey} for $d=2$. The resulting semifields will then be defined on the set $L^2$. 

\begin{construction} \label{thm:general_structure}
	Let $V=L^2$, $\mathbf{x},\mathbf{y}\in V$ with  $\mathbf{y}=\left(\begin{smallmatrix} y_0\\ y_1 \end{smallmatrix}\right)$ and $\sigma$ be a field automorphism of $L$ with fixed field $K$.  Further, 	let $T_a \in \Gamma L(2,L)$, $a \in L^*$ be irreducible transformations satisfying  $T_{a}+T_b=T_{a+b}$ for any $a,b \in L$
	where we set, with abuse of notation, $T_0=T_0^{-1}=0$. Then
	\[\mathbf{x} \circ \mathbf{y} = y_0 \mathbf{x}+ \eta T_{y_1}(\mathbf{x}) +\det(M_{T_{y_1}})^{\overline{\sigma}} T_{y_1}^{-1}(\mathbf{x})  \]

	defines a semifield for any $\eta \in L$ with $N_{L\colon K}(\eta)\neq 1$.
\end{construction}
\begin{proof}
	All semifield properties are elementary to prove except right-distributivity and the non-existence of zero divisors. Let us start with the distributivity. A straightforward calculation shows that if $M_{T_a}=\left(\begin{smallmatrix} c_1 & c_2 \\ c_3 & c_4 \end{smallmatrix}\right)$ then 
	$$T_{a}^{-1}(\mathbf{x})=\frac{1}{\det(M_{T_{a}})^{\overline{\sigma}}}\begin{pmatrix}
		c_4^{\overline{\sigma}} & -c_2^{\overline{\sigma}} \\
		-c_3^{\overline{\sigma}} & c_1^{\overline{\sigma}}
	\end{pmatrix} \mathbf{x}^{\overline{\sigma}}.$$

	So if $T_{a}+T_b=T_{a+b}$ then $\det(M_{T_a})^{\overline{\sigma}}T_{a}^{-1}+\det(M_{T_b})^{\overline{\sigma}}T_b^{-1}=\det(M_{T_{a+b}})^{\overline{\sigma}}T_{a+b}^{-1}$. 
	This immediately shows $\mathbf{x}\circ (\mathbf{y}+\mathbf{z})=\mathbf{x}\circ \mathbf{y}+\mathbf{x}\circ \mathbf{z}$ for all $\mathbf{x},\mathbf{y},\mathbf{z} \in V$.

	It thus only remains to check that the semifield has no zero divisors.
	
		Consider the right-multiplications $R_{\mathbf{y}}(\mathbf{x})$ for $\mathbf{y}\neq \mathbf{0}$. If $y_1=0$ then $R_{\mathbf{x}}(\mathbf{y})=y_0\mathbf{x}$ is non-singular for $y_0\neq 0$. In the case that $y_1\neq 0$ the $R_{\mathbf{y}}$ are non-singular by Corollary~\ref{cor:sheekey}.

\end{proof}

\begin{construction} \label{thm:general_structure_2}
	Let $V=L^2$, $\mathbf{x},\mathbf{y}\in V $ with $\mathbf{y}=\left(\begin{smallmatrix} y_0\\ y_1 \end{smallmatrix}\right)$. Let further $\sigma$ be a field automorphisms of $L$ with fixed field $K$.  Further, 	let $T_a \in \Gamma L(2,L)$, $a \in L^*$ be irreducible transformations satisfying  $T_{a}+T_b=T_{a+b}$ for any $a,b \in L$, where we set $T_0=0$. Then
	\[\mathbf{x} \circ \mathbf{y} = y_0 \mathbf{x}+ T_{y_1}(\mathbf{x}) \]

	defines a semifield.
\end{construction}
\begin{proof}
	The only non-obvious part to prove is again right distributivity and the non-existence of zero divisors. Right-distributivity follows immediately from the condition $T_{a}+T_b=T_{a+b}$ for any $a,b \in L$, and the right-multiplication mappings
	\[R_{\mathbf{y}}(\mathbf{x})=y_0 \mathbf{x}+ T_{y_1}(\mathbf{x})\]
	are non-singular by Theorem~\ref{thm:sheekeysingular} for $d=2$ and $y_d=0$.
\end{proof}

Constructions~\ref{thm:general_structure} and \ref{thm:general_structure_2} both rely on the existence of irreducible semilinear transformations $T_a$ satisfying $T_{a}+T_b=T_{a+b}$ for any $a,b \in L$. 

From a high level view, this means that an $m$-dimensional subspace in $\End_{\F_p}(L^2)$ of  irreducible transformation semilinear over $L$ can be used to construct a $2m$-dimensional subspace in $\End_{\F_p}(L^2)$ where every non-zero mapping is invertible (namely, the set of right-multiplications of a semifield with $p^{2m}$ elements).

Finding examples for such spaces of irreducible transformations will be the objective of the next section.

\section{Finding admissible mappings} \label{sec:admissible}
Both the constructions presented in the previous section need a certain mapping as a building block. 

\begin{definition}
We call a mapping $\mathcal{T} \colon L \rightarrow \Gamma L (2,L) \cup \{0\}$ admissible if
\begin{enumerate}
	\item $\mathcal{T}(0)=0$,
	\item $\mathcal{T}(a)+\mathcal{T}(b)=\mathcal{T}({a+b})$ for any $a,b \in L$,
	\item $\mathcal{T}(a) \in \Gamma L (2,L)$ is irreducible for all $a \in L^*$.
\end{enumerate}
\end{definition}

Each admissible mapping will immediately lead to semifields via Constructions~\ref{thm:general_structure} and \ref{thm:general_structure_2}. 

\subsection{The trivial admissible mapping}

Let us consider an obvious choice for $\mathcal{T}$ first. Choose an arbitrary, fixed irreducible transformation $T \in \Gamma L (2,L)$. Then define $\mathcal{T}(a)=aT$ for all $a \in L$. It is immediate that $\mathcal{T}$ is admissible. The resulting semifield multiplication via Construction~\ref{thm:general_structure} is then:
\[\mathbf{x} \circ \mathbf{y} = y_0 \mathbf{x}+ \eta y_1T(\mathbf{x}) +y_1^{\overline{\sigma}} \det(M_{T})^{\overline{\sigma}} T^{-1}(\mathbf{x}).  \]
A transformation $\mathbf{x} \mapsto T(\mathbf{x})$ yields the isotopic semifield multiplication
\[\mathbf{x} \circ \mathbf{y} = y_1^{\overline{\sigma}} \det(M_{T})^{\overline{\sigma}} \mathbf{x}+y_0 T(\mathbf{x})+ \eta y_1T^2(\mathbf{x})  \]
which is a twisted cyclic semifield after an isotopism mapping $y_1 \mapsto y_1^{\sigma}$. A comparison with the twisted cyclic fields in Theorem~\ref{thm:sheekey} shows that, after an additional twist that replaces $\eta y_1T^2(\mathbf{x})$ above with $\eta y_1^\tau T^2(\mathbf{x})$ for a suitably chosen $\tau$ yields the same constraints as in Theorem~\ref{thm:sheekey}. This construction thus gives \emph{exactly} the twisted cyclic semifields for $d=2$. We can thus view the twisted cyclic fields with parameter $d=2$ as a "trivial" case of our construction.

Similarly, the construction using Construction~\ref{thm:general_structure_2} just yields
\[\mathbf{x} \circ \mathbf{y} = y_0\mathbf{x}+y_1 T(\mathbf{x}),  \]
which is exactly a cyclic semifield. 

We conclude that this "trivial" admissible mapping does not yield any new examples.

\subsection{A characterization of two dimensional irreducible transformations}

In order to find new, non-trivial admissible mappings, we need to classify when semilinear mappings in $\Gamma L(2,L)$ are irreducible. Note that complete classifications of irreducible semilinear transformations exist, see~\cite{dempwolff2010irreducible}. The complete classification is however less concrete and assumes a certain "normal form". We however only require the comparatively simple case of transformations over a two-dimensional space. We thus give a simple, direct proof that gives an easy to check criterion.

\begin{proposition} \label{prop:irred}
	The transformation $T \in \Gamma L (2,L)$ with associated $M_T=\left(\begin{smallmatrix} 0 & \alpha \\ 1 & \beta \end{smallmatrix}\right)\in \GL(2,L)$ and field automorphism $\sigma$ is irreducible if and only if the polynomial $P(X)=X^{\sigma+1}-\beta X-\alpha \in L[X]$ has no roots in $L$.
\end{proposition}
\begin{proof}
	Let $\mathbf{0}\neq \mathbf{x}=\left(\begin{smallmatrix} x_0\\ x_1 \end{smallmatrix}\right) $ and assume $\langle \mathbf{x} \rangle$ is a 1-dimensional invariant subspace, i.e. $T(\mathbf{x})=k\mathbf{x}$ for some $k \in L^*$. Then
	\[ \begin{pmatrix}
		\alpha x_1^\sigma \\
		x_0^\sigma+\beta x_1^\sigma
	\end{pmatrix}= \begin{pmatrix}
		kx_0 \\
		kx_1
	\end{pmatrix}.\]
	Note that $x_0=0$ if and only if $x_1=0$, so let $x_0x_1\neq 0$. Eliminating $x_0$ from the second equation yields
	\[\left(\frac{\alpha}{k}\right)^\sigma x_1^{\sigma^2}+\beta x_1^\sigma-kx_1=0.\]
	Setting $z=x_1^{\sigma-1}$ this is equivalent to
		\[\left(\frac{\alpha}{k}\right)^\sigma z^{\sigma+1}+\beta z-k=0,\]
		and then setting further $z'=-\alpha z/k$ shows that this equation has no solution if and only $z'^{\sigma+1}-\beta z'-\alpha=0$ has none. So $T$ is irreducible if $P(X)$ has no root in $L$. If $z'^{\sigma+1}-\beta z'-\alpha=0$ has a solution, then we can choose a $k$ such that $z=-kz'/\alpha$ is a $(\sigma-1)$-st power, allowing us to track back the steps, and find $x_0,x_1$ satisfying $T(\mathbf{x})=k\mathbf{x}$, and thus constructing a non-trivial invariant subspace.
\end{proof}

Note that the polynomials $P(X)=X^{\sigma+1}-\beta X-\alpha \in L[X]$ have been extensively studied under the name of projective polynomials~\cite{abhyankar1997projective,Bluher}, for instance the number of projective polynomials with no roots in a finite field $L$, and thus the number of irreducible semilinear transformations in $\Gamma L (2,L)$ of the form in Proposition~\ref{prop:irred}, was explicitly determined in~\cite{Bluher}. In particular projective polynomials with no roots always exist, independent of the choice of $L$ and $\sigma$.

Note also that we only need to consider irreducible semilinear transformations in $\Gamma L (2,L)$ up to $\Gamma L (2,L)$-conjugacy since conjugate mappings will yield isotopic semifields in our constructions. This has been observed in the case of cyclic semifields by Kantor and Liebler~\cite{kantor2008semifields}, and the same reasoning still applies here. It is then clear that we can always choose a representative of the form in Proposition~\ref{prop:irred}.

%

\subsection{Some nontrivial admissible mappings}

We now present two nontrivial admissible mappings. We will need a simple and well known elementary lemma.
	\begin{lemma} \label{lem:gcd}
	Let $p$ be a prime, $k,l\geq 1$, $t=\gcd(k,l)$. Then 
	\begin{align*}
	\gcd(p^k-1,p^l-1)&=p^{t}-1,\\
	\gcd(p^k+1,p^l-1)&=\begin{cases}
				p^{d}-1, & l/t \text{ even },\\
				2, & l/t \text{ odd}, p>2, \\
				1,& l/t \text{ odd}, p=2,
			\end{cases}\\
	\gcd(p^k+1,p^l+1)&=\begin{cases}
				p^{d}+1, & l/t \text{ odd  and }k/t \text{ odd},\\
				2, & l/t \text{ even  or }k/t \text{ even, and }p>2, \\
				1, & l/t \text{ even  or }k/t \text{ even, and }p=2.
				\end{cases}
				\end{align*}
	\end{lemma}

\begin{proposition}\label{prop:admissible1}
	Define $\mathcal{T} \colon L \rightarrow \Gamma L (2,L)\cup\{0\}$ such that $\mathcal{T}(0)=0$ and for $a \neq 0$ define $\mathcal{T}(a) \in \Gamma L(2,L)$ with associated field automorphism $\sigma$ and associated matrix $M_a \in \GL(2,L)$ via
	\[M_a = \begin{pmatrix}
		0 & a \alpha\\
		a^{\tau} & 0
	\end{pmatrix}\]
	for an arbitrary, nontrivial field automorphism $\tau$. Write $\sigma \colon x \mapsto x^{p^k}$, $\tau \colon x \mapsto x^{p^l}$, $0\leq k,l<m$. Then $\mathcal{T}$ is admissible  if and only if either
	\begin{itemize}
		\item $\alpha$ is a non-square; and $k=0$ or $\gcd(m,l)/\gcd(m,k,l)$ is odd; or
		\item $k\neq 0$, $\alpha$ is not a $(p^{\gcd(m,k,l)}+1)$-st power and $\gcd(m,l)/\gcd(m,k,l)$ is even.
	\end{itemize}
\end{proposition}
	\begin{proof}
	$\mathcal{T}(a)+\mathcal{T}(b)=\mathcal{T}({a+b})$ holds for any $a,b \in L$ by construction. \\
	By Proposition~\ref{prop:irred}, $\mathcal{T}(a)$ is irreducible if and only if $a^\tau X^{\sigma+1}-a\alpha\neq 0$ has no solutions in $L$, this is equivalent to $X^{\sigma+1}\neq a^{1-\tau}\alpha $ having no solutions. Let $G=L^{\sigma+1}L^{\tau-1}$. Then $\mathcal{T}(a)$ is irreducible for all $a \in L^*$ if and only if $\alpha \notin G$. $G$ is precisely the set of $(\gcd(p^k+1,p^l-1,p^m-1))$-th powers. The result then follows from Lemma~\ref{lem:gcd}.
	
	
	%
	%
	%
	\end{proof}
	Note that the condition in Proposition~\ref{prop:admissible1} in particular always holds if $p$ is odd and $\alpha$ is a non-square. Furthermore, we can allow $\tau=\id$ in Proposition~\ref{prop:admissible1}, but then $M_a=a\left(\begin{smallmatrix}
		0 &  \alpha\\
		1 & 0
	\end{smallmatrix}\right)$ and the resulting admissible mappings are trivial admissible mappings for the fixed irreducible transformation $T \in \Gamma L (2,L)$ with $M_T=\left(\begin{smallmatrix}
		0 &  \alpha\\
		1 & 0
	\end{smallmatrix}\right)$.
	
	We now present another nontrivial admissible mapping:
	\begin{proposition} \label{prop:admissible2}
	Define $\mathcal{T} \colon L \rightarrow \Gamma L (2,L)\cup\{0\}$ such that $\mathcal{T}(0)=0$ and for all $a \neq 0$ let $\mathcal{T}(a) \in \Gamma L(2,L)$ with associated field automorphism $\sigma$ via
	\[M_a = \begin{pmatrix}
		0 & a \alpha\\
		a^{\sigma^2} & a^\sigma \beta
	\end{pmatrix}.\]
	Then 
	$\mathcal{T}$
	is admissible if and only if $P(X)=X^{\sigma+1}-\beta X -\alpha \in L[X]$ has no roots in $L$.
\end{proposition}
	\begin{proof}
	Again, $\mathcal{T}(a)+\mathcal{T}(b)=\mathcal{T}({a+b})$ holds for any $a,b \in L$ by construction, and again by Proposition~\ref{prop:irred}, $\mathcal{T}(a)$ is irreducible if and only if $a^{\sigma^2}X^{\sigma+1}-a^\sigma \beta X - a \alpha \in L[X]$ has no roots in $L$. A transformation $X \mapsto X/(a^{\sigma-1})$ leads to $aX^{\sigma+1}-a \beta X- a \alpha \in L[X]$ which clearly has no roots if and only if $P(X)$ has no roots.
	\end{proof}

	\begin{remark}
		Note that the admissible mappings constructed in Propositions~\ref{prop:admissible1} and \ref{prop:admissible2} coincide if we choose $\tau=\sigma^2$ in Proposition~\ref{prop:admissible1} and $\beta=0$ in Proposition~\ref{prop:admissible2}. 
	\end{remark}
	
\section{Recognizing known constructions} \label{sec:known}

In this section, we consider known semifields families that produce semifields of size $p^{2m}$ for infinitely many $p,m$. Most of these constructions were found using ad hoc methods, and the relations between them was unclear. Our framework will explain many of them. 

Table~\ref{t:1} gives a succinct list of known constructions that can be explained using our constructions.

Note that this explicitly excludes constructions of Knuth-Kantor~\cite{Kantor03,KW}, Cohen-Ganley~\cite{CG}, and Ganley~\cite{Ganley}, that produce semifields only in characteristic $2$ or $3$. We also exclude the finite fields (which are trivially semifields) as well as Albert's twisted field~\cite{Albert61}, which yield semifields of order $p^n$ for any $n$, i.e., also in odd extensions. 

As the table shows, our constructions cover (depending on one wants to count) around ten earlier infinite families of semifields. We want to emphasize that our constructions thus cover the oldest known family of semifields (the Dickson semifields), the semifields with $2$ nuclei that are as large as possible~\cite{Knuth65} (the Knuth semifields), as well as the largest known family of inequivalent semifields in odd characteristic overall~\cite{golouglu2023exponential} (the Taniguchi semifields).  The remainder of this section is devoted to proving the contents of Table~\ref{t:1}.

\begin{table}[!ht] 
\noindent\begin{center} 
{\scriptsize
\begin{tabular}{|c|c|c|c|c|c|} 
\hline 
\textbf{Family} &  \textbf{Construction} & \textbf{Admissible Mapping}  &  \textbf{Reference} & \textbf{Notes}\\
\hline 
(Generalized) Dickson & Construction~\ref{thm:general_structure_2} &Proposition~\ref{prop:admissible1} & \cite{dickson1906commutative,Knuth65} & ---\\
\hline 
Knuth I & Construction~\ref{thm:general_structure_2} & Proposition~\ref{prop:admissible2} &  \cite{Knuth65} &---\\
\hline 
Knuth II,III,IV, & Construction~\ref{thm:general_structure_2} & trivial & \cite{Knuth65}&--- \\
Hughes-Kleinfeld & && \cite{hughes1960seminuclear} &\\
\hline 
Bierbrauer, & Construction~\ref{thm:general_structure} & trivial & \cite{bartoli2017family,Bierbrauer16,BH} &Contains commutative semifields\\
Budaghyan-Helleseth & & & & \\
\hline 
Dempwolff & Construction~\ref{thm:general_structure} & trivial & \cite{dempwolff2013more} & Unnecessary conditions in~\cite{dempwolff2013more}\\
\hline 
Zhou-Pott & Construction~\ref{thm:general_structure}  &Proposition~\ref{prop:admissible1} & \cite{ZP13} & For parameter $\eta=-1$ \\
\hline
Taniguchi &  Construction~\ref{thm:general_structure}  &Proposition~\ref{prop:admissible2} & \cite{taniguchisf} &Largest known construction for $p$ odd \\
\hline
(Twisted) cyclic semifields & Constructions~\ref{thm:general_structure},~\ref{thm:general_structure_2} & trivial & \cite{sheekey2020new} &--- \\
\hline
\end{tabular} 
}
\end{center}
\caption{Known infinite families of semifields of order $p^{2m}$ and how to recreate them using Constructions~\ref{thm:general_structure} and \ref{thm:general_structure_2}} \label{t:1}
\end{table}

\subsection{Generalized Dickson semifields}
The first proper family of semifields (i.e. semifields that are not isotopic to fields) that was found was the family of \emph{Dickson semifields} in 1906~\cite{dickson1906commutative}. It was generalized by Knuth in 1965~\cite{Knuth65} to the following family, see also~\cite[p.241]{dembowski1997finite}. 	Let $\sigma,\tau,\rho$ be field automorphisms of $L$ defined by $\sigma \colon x \mapsto x^{p^k}$, $\tau \colon x \mapsto x^{p^l}$, $\rho \colon x \mapsto x^{p^r}$. Let further $\alpha \notin L^{p^r+1}L^{p^k+1}L^{p^t-1}$. Then 
\[\begin{pmatrix}
	x_0 \\
	x_1
\end{pmatrix} \circ_D\begin{pmatrix}
	y_0 \\
	y_1
\end{pmatrix} = \begin{pmatrix}
	x_0y_0+\alpha x_1^\rho y_1^\tau\\ x_0^\sigma y_1+x_1y_0
\end{pmatrix}=\begin{pmatrix}
	f(x_0,x_1,y_0,y_1)\\ g(x_0,x_1,y_0,y_1)
\end{pmatrix}\]
defines a semifield $\S$. The original semifields by Dickson choose $\sigma=\id$ and $\rho=\tau$ and are the only commutative semifields in the family. 


We now compute the transpose $\pi_2(\S)$ of $\S$ by computing the dual spread. This is a pure computation that we need multiple times throughout the paper; we will explain the steps once. We roughly follow a technique introduced by Kantor~\cite{Kantor03}, see also~\cite{ball2004six,ball2007hughes}. 

The elements from the spread defined by the semifield are 
\[ A_{y_0,y_1}=\{(x_0,x_1,f(x_0,x_1,y_0,y_1),g(x_0,x_1,y_0,y_1)) \colon x_0,x_1 \in L\}\]
for $y_0,y_1 \in L$ and $\{(0,0,x_0,x_1)\colon x_0,x_1 \in L\}$. We use the alternating form 
\[((u,v,w,z)(a,b,c,d))=\Tr(cu+dv-aw-bz)\]
to determine the dual spread. Here $\Tr$ denotes the absolute trace. The dual spread is then the union of the subspaces
\[A^\perp_{y_0,y_1}=\{(a,b,c,d) \colon \Tr(cx_0+dx_1-af(x_0,x_1,y_0,y_1)-bg(x_0,x_1,y_0,y_1))=0 \text{ for all } x_0,x_1 \in L\}\]
for $y_0,y_1 \in L$ together with $\{(0,0,x_0,x_1)\colon x_0,x_1 \in L\}$.
We evaluate $\Tr(cx_0+dx_1-af(x_0,x_1,y_0,y_1)-bg(x_0,x_1,y_0,y_1))=0$ for $x_0=0$:
\begin{align*}
	0&=\Tr(dx_1-af(0,x_1,y_0,y_1)-bg(0,x_1,y_0,y_1)) \\
&=\Tr(dx_1-a(\alpha x_1^\rho y_1^\tau)-bx_1y_0) \\
&=\Tr(x_1(d-(a\alpha)^{\overline{\rho}}y_1^{\tau \overline{\rho}}-by_0)). 
\end{align*} 
Since this has to hold for all $x_1\in L$, this necessitates 
\[d=(a\alpha)^{\overline{\rho}}y_1^{\tau \overline{\rho}}+by_0.\]
Similarly, evaluating $\Tr(cx_0+dx_1-af(x_0,x_1,y_0,y_1)-bg(x_0,x_1,y_0,y_1))=0$ for $x_1=0$ yields:
\begin{align*}
	0&=\Tr(cx_0-af(x_0,0,y_0,y_1)-bg(x_0,0,y_0,y_1)) \\
&=\Tr(cx_0-ax_0y_0-bx_0^\sigma y_1) \\
&=\Tr(x_0(c-ay_0-(by_1)^{\overline{\sigma}})). 
\end{align*} 
Again this leads to
\[c=ay_0+(by_1)^{\overline{\sigma}}.\]
So the dual spread is constituted by the elements
\[A^\perp_{y_0,y_1}=\{(a,b,ay_0+(by_1)^{\overline{\sigma}},(a\alpha)^{\overline{\rho}}y_1^{\tau \overline{\rho}}+by_0)\colon a,b \in L\}.\] 
The corresponding transposed semifield $\pi_2(\S)$ is then defined by 
\[\begin{pmatrix}
	x_0 \\
	x_1
\end{pmatrix} \circ_D^t\begin{pmatrix}
	y_0 \\
	y_1
\end{pmatrix} = \begin{pmatrix}
	x_0y_0+(x_1y_1)^{\overline{\sigma}}\\ (\alpha x_0)^{\overline{\rho}}y_1^{\tau \overline{\rho}}+x_1y_0
\end{pmatrix}.\]

We can spot some intersections with the semifields we constructed via Construction~\ref{thm:general_structure_2} using the admissible mapping from Proposition~\ref{prop:admissible1}. Indeed, if $\rho=\sigma$, we have
\[\begin{pmatrix}
	x_0 \\
	x_1
\end{pmatrix} \circ_D^t\begin{pmatrix}
	y_0 \\
	y_1
\end{pmatrix} = y_0\begin{pmatrix}
	x_0\\x_1
\end{pmatrix} + \begin{pmatrix}
	0 & y_1^{\overline{\sigma}} \\
	\alpha^{\overline{\sigma}} y_1^{\tau \overline{\sigma}} & 0 
\end{pmatrix}\begin{pmatrix}
	x_0^{\overline{\sigma}}\\x_1^{\overline{\sigma}}
\end{pmatrix}.\]

This is clearly (after transformations $y_1 \mapsto y_1^\sigma$  and dividing by $\alpha^{\overline{\sigma}}$) isotopic to Construction~\ref{thm:general_structure_2} using the admissible mapping from Proposition~\ref{prop:admissible1}. Moreover, the original Dickson semifield ($\sigma=\id, \tau=\rho$) yields 
\[\begin{pmatrix}
	x_0 \\
	x_1
\end{pmatrix} \circ_D^t\begin{pmatrix}
	y_0 \\
	y_1
\end{pmatrix} = x_1\begin{pmatrix}
	y_1\\y_0
\end{pmatrix} + \begin{pmatrix}
	0 & x_0 \\
	\alpha^{\overline{\rho}} x_0^{\overline{\rho}} & 0 
\end{pmatrix}\begin{pmatrix}
	y_1\\y_0
\end{pmatrix},\]
so it is again isotopic to semifields from Construction~\ref{thm:general_structure_2} with an admissible mapping from Proposition~\ref{prop:admissible1}. 

Summarizing, we observe that Construction~\ref{thm:general_structure_2} with Proposition~\ref{prop:admissible1} is (via the Knuth orbit) contained in the generalized Dickson semifields. In particular, our construction contains the original Dickson semifield found in 1906.

\subsection{Knuth' semifields quadratic over a weak nucleus}

In 1965, Knuth~\cite{Knuth65} provided four families of semifields of order $p^{2m}$, the so-called \emph{semifields quadratic over a weak nucleus}. 

\begin{definition}[The Knuth semifields I - IV] \label{def:Knuth}
	Let $V=L^2$. Let $\id \neq \sigma$ be a field automorphism of $L$  and $P(X)=X^{\sigma+1}-\beta X-\alpha \in L[X]$ be a polynomial with no roots in $L$.  The following four operations then yield semifield multiplications on $V$:
	\begin{enumerate}[label=\Roman*]
		\item $\begin{pmatrix}x_0 \\ x_1 \end{pmatrix} \circ_{K_1} \begin{pmatrix}y_0 \\ y_1 \end{pmatrix} = \begin{pmatrix}x_0y_0+\alpha x_1^{\sigma} y_1^{\overline{\sigma}^2}\\ x_1y_0+x_0^{{\sigma}} y_1+\beta x_1^{{\sigma}} y_1^{\overline{\sigma}}\end{pmatrix}$,
		\item $\begin{pmatrix}x_0 \\ x_1 \end{pmatrix} \circ_{K_2} \begin{pmatrix}y_0 \\ y_1 \end{pmatrix} =\begin{pmatrix}x_0y_0+\alpha x_1^{{\sigma}} y_1\\  x_1y_0+x_0^{{\sigma}} y_1+\beta x_1^{{\sigma}} y_1\end{pmatrix} $,
		\item $\begin{pmatrix}x_0 \\ x_1 \end{pmatrix} \circ_{K_3} \begin{pmatrix}y_0 \\ y_1 \end{pmatrix} =\begin{pmatrix}x_0y_0+\alpha x_1^{\overline{\sigma}} y_1^{\overline{\sigma}^2}\\  x_1y_0+x_0^{{\sigma}} y_1+\beta x_1 y_1^{\overline{\sigma}}\end{pmatrix}$,
		\item $\begin{pmatrix}x_0 \\ x_1 \end{pmatrix} \circ_{K_4} \begin{pmatrix}y_0 \\ y_1 \end{pmatrix} =\begin{pmatrix}x_0y_0+\alpha x_1^{\overline{\sigma}} y_1\\  x_1y_0+x_0^{{\sigma}} y_1+\beta x_1 y_1\end{pmatrix}.$
	\end{enumerate}
\end{definition}
The Knuth II semifields were found already in 1960 and are often also called the Hughes-Kleinfeld semifields~\cite{hughes1960seminuclear}. It is well known (and confirmed by a simple calculation) that the Hughes-Kleinfeld/Knuth II semifields are actually just cyclic semifields. In particular, we can view them also as the semifields constructed using Construction~\ref{thm:general_structure} using the trivial admissible mapping. Knuth III and IV semifields are in the same Knuth orbit as (and thus equivalent to) the Knuth II semifields, as shown in~\cite{ball2007hughes}.

It remains to investigate the Knuth I semifields. It is immediate to verify that the Knuth I multiplication is (after taking $y_1 \mapsto y_1^{\sigma^2}$) precisely
\[\begin{pmatrix}
	x_0 \\
	x_1
\end{pmatrix} \circ_{K_1}\begin{pmatrix}
	y_0 \\
	y_1
\end{pmatrix} = y_0\begin{pmatrix}
	x_0\\x_1
\end{pmatrix} + \begin{pmatrix}
	0 & y_1\alpha \\
	y_1^{\sigma^2} & y_1^\sigma \beta 
\end{pmatrix}\begin{pmatrix}
	x_0^\sigma\\x_1^\sigma
\end{pmatrix},\]
which is exactly Construction~\ref{thm:general_structure} using the admissible mapping from Proposition~\ref{prop:admissible2}.\\

We conclude that all Knuth semifields quadratic over a weak nucleus are covered by our constructions.

\subsection{Bierbrauer semifields}

We now consider the semifields introduced by Bierbrauer~\cite{Bierbrauer16} in 2016 for odd characteristic, with generalizations to even characteristic in~\cite{bartoli2017family}. The semifield is again defined on $L^2$ where $L$ is a finite field and $\sigma$ is a field autormorphism of $L$. Then
\begin{equation*}
	\begin{pmatrix}x_0 \\ x_1 \end{pmatrix} \circ_{B} \begin{pmatrix}y_0 \\ y_1 \end{pmatrix} = \begin{pmatrix}\delta x_1y_1^\sigma+\gamma x_0y_1^\sigma + \beta x_1y_0^\sigma + \alpha x_0y_0^\sigma + \eta(\delta x_1^\sigma y_1-\gamma x_1^\sigma y_0-\beta x_0^\sigma y_1 + \alpha x_0^\sigma y_0) \\
	x_0y_1+x_1y_0\end{pmatrix}
\end{equation*}
 is a semifield if $\eta$ is not a $(\sigma-1)$-st power and $P(X)=\delta X^{\sigma +1} + \gamma X^\sigma - \beta X-\alpha \in L[X]$ has no zero in $L$ with $\delta \neq 0$. By scaling, we can assume without loss of generality that $\delta=1$. Further, by transformations $x_1 \mapsto x_1-\gamma x_0$, $y_1 \mapsto y_1+ \gamma y_0$ we reach a semifield with parameter $\gamma=0$. If $\beta \neq 0$ we can further do a transformation $x_1 \mapsto x_1\beta^{\overline{\sigma}}$, $y_1 \mapsto y_1\beta^{\overline{\sigma}}$ to achieve $\beta=1$. Summarizing, we only need to consider $\delta=1,\gamma=0,\beta \in \{0,1\}$, resulting in:

\begin{equation}
	\begin{pmatrix}x_0 \\ x_1 \end{pmatrix} \circ_{B} \begin{pmatrix}y_0 \\ y_1 \end{pmatrix} = \begin{pmatrix}x_1y_1^\sigma+ \beta x_1y_0^\sigma + \alpha x_0y_0^\sigma + \eta( x_1^\sigma y_1-\beta x_0^\sigma y_1 + \alpha x_0^\sigma y_0) \\
	x_0y_1+x_1y_0\end{pmatrix}.
\label{eq:bb2}
\end{equation}

We again compute the transpose of this semifield using the same technique we showed earlier and arrive at

\begin{equation}
		\begin{pmatrix}x_0 \\ x_1 \end{pmatrix} \circ_{B}^t \begin{pmatrix}y_0 \\ y_1 \end{pmatrix} = \begin{pmatrix} \alpha x_0y_0^\sigma + \eta^{\overline{\sigma}}(-\beta^{\overline{\sigma}}  (x_0y_1)^{\overline{\sigma}} + \alpha^{\overline{\sigma}} (x_0y_0)^{\overline{\sigma}})+x_1y_1 \\
	x_0y_1^\sigma + \beta x_0y_0^\sigma + \eta^{\overline{\sigma}} (x_0y_1)^{\overline{\sigma}}+x_1y_0 \end{pmatrix}.
\label{eq:trans_BB}
\end{equation}
Let now $\mathbf{y}=\left(\begin{smallmatrix}y_1 \\ y_0\end{smallmatrix}\right)$ and $T \in \Gamma L (2, L)$ with associated field automorphism $\sigma$ and $M_T=\left(\begin{smallmatrix}	0 & \alpha \\
	1 & \beta\end{smallmatrix}\right)$.
Then Eq.~\eqref{eq:trans_BB} is exactly
\begin{align*}
\begin{pmatrix}x_0 \\ x_1 \end{pmatrix} \circ_{B}^t \begin{pmatrix}y_0 \\ y_1 \end{pmatrix} &= x_1 \begin{pmatrix}
	y_1 \\ y_0
\end{pmatrix}+x_0\begin{pmatrix}
	0 & \alpha \\
	1 & \beta
\end{pmatrix} \begin{pmatrix}
	y_1^{{\sigma}} \\ y_0^{{\sigma}}
\end{pmatrix}+
x_0\eta^{\overline{\sigma}}  \begin{pmatrix}
	-\beta^{\overline{\sigma}} & \alpha^{\overline{\sigma}} \\
	1 & 0
\end{pmatrix} \begin{pmatrix}
	y_1^{\overline{\sigma}} \\ y_0^{\overline{\sigma}} 
\end{pmatrix} \\&= x_1 \mathbf{y} + x_0T(\mathbf{y}) + x_0 \eta^{\overline{\sigma}} \det(M_T)^{\overline{\sigma}}  T^{-1} (\mathbf{y}).
\end{align*}

It is easy to see that (after minor transformations) this is exactly the dual of  Construction~\ref{thm:general_structure} using the trivial admissible mapping. As our discussions on the trivial admissible mapping show, this means in particular that the Bierbrauer semifields are contained (via Knuth orbit and isotopy) in the family of twisted cyclic semifields. This has interesting consequences: It was shown~\cite{Bierbrauer16} that the Bierbrauer semifields contain a family of \emph{commutative} semifields, namely the Budaghyan-Helleseth semifields. This means that the twisted cyclic semifields (via Knuth orbit) contain these commutative semifields as well. Commutative semifields have received special attention: Few constructions are known, and (outside of purely algebraic interest), they lead to the construction of symplectic spread sets (see e.g.~\cite{Kantor03}), rank-metric codes with symmetric matrices, as well as so called planar or perfect nonlinear functions (see~\cite{coulterhenderson}) which have applications in cryptography~\cite{blondeau2015perfect}. Finding commutative semifields is not easy: Since isotopy does not preserve commutativity, it is generally difficult to determine if a family of semifields is isotopic to a \emph{commutative} one. Locating the commutative Budaghyan-Helleseth semifields in this non-trivial way in the family of twisted cyclic semifields leads to the following natural question:

\begin{problem}\label{prob}
	Investigate if the twisted cyclic fields contain more commutative semifields (via isotopy and Knuth orbit).
\end{problem}

\subsection{Dempwolff semifields}

We now consider the semifields introduced by Dempwolff~\cite[Theorem B]{dempwolff2013more} in 2013. The semifield is defined on $L^2$ where $L$ is a finite field of odd characteristic, $\sigma$ is a field autormorphism of $L$ with fixed field $K$, $\tau$ a field automorphism with fixed field $K'\leq K$, $\alpha$ is a non-square and $\alpha$ satisfies $N_{L\colon K}(\eta)\neq 1$. Then
\begin{equation*}
\begin{pmatrix}x_0 \\ x_1 \end{pmatrix} \circ_{De} \begin{pmatrix}y_0 \\ y_1 \end{pmatrix} = \begin{pmatrix} x_0y_0+x_1y_1^\sigma-\eta x_1^\tau y_1^{\overline{\sigma}} \\ x_0y_1+\alpha (x_1y_0^\sigma-\eta x_1^\tau y_0^{\overline{\sigma}}) \end{pmatrix}
\end{equation*}
is a semifield. This can easily be rewritten as, setting $\mathbf{y}=\left(\begin{smallmatrix}y_0 \\ y_1\end{smallmatrix}\right)$
\begin{equation*}
\begin{pmatrix}x_0 \\ x_1 \end{pmatrix} \circ_{De} \begin{pmatrix}y_0 \\ y_1 \end{pmatrix} = x_0 \mathbf{y}+x_1 T(\mathbf{y})+\eta x_1^\tau T^{-1}(\mathbf{y})
\end{equation*}
for $T \in \Gamma L(2,L)$ with associated field automorphism $\sigma$ and $M_T=\left(\begin{smallmatrix}	0 & 1 \\
	\alpha & 0\end{smallmatrix}\right)$. 
This is covered by Construction~\ref{thm:general_structure} with the trivial admissible mapping, after the additional twist we explained in Section~\ref{sec:admissible}.  This shows again the strength of our framework: $T$ could have been chosen more generally as \emph{any} irreducible semilinear transformation in $\Gamma L(2,L)$ (see our classification in Section~\ref{sec:admissible}), instead of the simple choice here. In particular, there is no reason to restrict ourselves to odd characteristic, and some further unnecessary restrictions in~\cite[Theorem B]{dempwolff2013more} on the order of $\sigma$ can also be removed.

\subsection{Zhou-Pott's commutative semifields}

The following family of commutative semifields was found by Zhou and Pott in 2013~\cite{ZP13}. Let $L$ be a field of odd characteristic; and $\sigma,\tau$ be field automorphisms of $L$ where $\sigma$ has odd order. Let further $\alpha$ be a non-square. Then
\begin{equation*}
\begin{pmatrix}x_0 \\ x_1 \end{pmatrix} \circ_{ZP} \begin{pmatrix}y_0 \\ y_1 \end{pmatrix}  = \begin{pmatrix} x_0^\sigma y_0+x_0y_0^\sigma +\alpha (x_1^\sigma y_1+x_1y_1^\sigma )^\tau \\
	x_0y_1+x_1y_0 \end{pmatrix},
\end{equation*}
defines a (commutative) semifield multiplication. We compute the transpose of this semifield (again using the techniques described earlier):

\begin{equation}
\begin{pmatrix}x_0 \\ x_1 \end{pmatrix} \circ_{ZP}^t \begin{pmatrix}y_0 \\ y_1 \end{pmatrix}  = \begin{pmatrix}x_0^{\overline{\sigma}}y_0^{\overline{\sigma}}+x_0y_0^\sigma +x_1y_1 \\
\alpha^{ \overline{\tau\sigma}}x_0^{\overline{\tau\sigma}}y_1^{\overline{\sigma}}+\alpha^{\overline{\tau}}x_0^{\overline{\tau}}y_1^{\sigma}+x_1y_0\end{pmatrix}
\label{eq:trans}
\end{equation}

We rewrite this, setting $\mathbf{y}=\left(\begin{smallmatrix}y_1 \\ y_0\end{smallmatrix}\right)$,

\begin{equation}
\begin{pmatrix}x_0 \\ x_1 \end{pmatrix} \circ_{ZP}^t \begin{pmatrix}y_0 \\ y_1 \end{pmatrix}   = x_1\mathbf{y} - T_{-x_0}(\mathbf{y}) + \det(M_{T_{-x_0}})^{\overline{\sigma}}T^{-1}_{-x_0}(\mathbf{y}) 
\label{eq:zp}
\end{equation}

where $T_{x_0}\in \Gamma L (2,L)$ for $x_0 \neq 0$ has associated field automorphism $\sigma$ and 
\[M_{T_{x_0}} = \begin{pmatrix}
	 0 & x_0 \\
	(\alpha x_0 )^{\overline{\tau}} & 0
\end{pmatrix}.\]
So the transpose of the Zhou-Pott semifields is isotopic to Construction~\ref{thm:general_structure} with the admissible mapping from Proposition~\ref{prop:admissible1} and $\eta=-1$. It is easy to check that the conditions of Construction~\ref{thm:general_structure} then coincide with the conditions of the Zhou-Pott semifields. In Section~\ref{sec:new} we will investigate the semifields that arise when we choose $\eta \neq -1$. In particular, this allows us to find new semifields in even characteristic. 

\subsection{Taniguchi's semifields}

In 2019, Taniguchi~\cite{taniguchisf} provided the following family of semifields. Let $L$ be a finite field, $\sigma$ be a field automorphism, $\eta$ is not $(\sigma-1)$-st power and $X^{\sigma+1}-\beta' X-\alpha' \in L[X]$ has no root in $L$. Then
\[\begin{pmatrix}x_0 \\ x_1 \end{pmatrix} \circ_{T} \begin{pmatrix}y_0 \\ y_1 \end{pmatrix}  = \begin{pmatrix}
	(x_0^\sigma y_0-\eta x_0y_0^\sigma)^{\sigma^2}+\beta'(x_0^\sigma y_1+\eta y_0^\sigma x_1)^\sigma+\alpha'(x_1^\sigma y_1-\eta x_1y_1^\sigma) \\
	x_1y_0+x_0y_1
\end{pmatrix}\]
defines a semifield multiplication. We first apply  $\overline{\sigma^2}$ to the first component and get

\[\begin{pmatrix}x_0 \\ x_1 \end{pmatrix} \circ_{T'} \begin{pmatrix}y_0 \\ y_1 \end{pmatrix}  = \begin{pmatrix} x_0^\sigma y_0-\eta x_0y_0^\sigma+\beta(x_0^\sigma y_1+\eta y_0^\sigma x_1)^{\overline{\sigma}}+\alpha(x_1^\sigma y_1-\eta x_1y_1^\sigma)^{\overline{\sigma}^2} \\
x_1y_0+x_0y_1   \end{pmatrix},\]
where $\alpha^{\sigma^2}=\alpha'$ and $\beta^{\sigma^2}=\beta'$. Note that then $X^{\sigma+1}-\beta X-\alpha \in L[X]$ has also no root in $L$. We now compute the transpose as for the previous cases. Setting $\mathbf{y}=\left(\begin{smallmatrix}y_1 \\ y_0\end{smallmatrix}\right)$, we obtain:

\begin{align*}
	\begin{pmatrix}x_0 \\ x_1 \end{pmatrix} \circ_{T'}^t \begin{pmatrix}y_0 \\ y_1 \end{pmatrix} &=  \begin{pmatrix} x_1^{\overline{\sigma}}y_0^{\overline{\sigma}}-\eta x_0y_0^\sigma +\beta x_0 y_1^{\overline{\sigma}}+x_1y_1 \\	
	\eta \beta^\sigma x_0^\sigma y_0^{\sigma}+\alpha^{\sigma}x_0^{\sigma}y_1^{\overline{\sigma}}-\alpha^{\sigma^2}\eta x_0^{\sigma^2} y_1^{\sigma}+x_1y_0 \end{pmatrix} \\
	&=x_1 \mathbf{y} + \eta \begin{pmatrix}
		0 & - x_0 \\
		-(\alpha x_0)^{\sigma^2} & (\beta x_0)^\sigma
	\end{pmatrix}\mathbf{y}^\sigma + \begin{pmatrix}
		\beta x_0 & x_0^{\overline{\sigma}} \\
		(\alpha x_0)^\sigma &0
	\end{pmatrix} \mathbf{y}^{\overline{\sigma}}.
\end{align*}
A transformation $x_0 \mapsto -x_0/\alpha$ yields 
\begin{align*}
	\begin{pmatrix}x_0 \\ x_1 \end{pmatrix} \ast \begin{pmatrix}y_0 \\ y_1 \end{pmatrix}
	&=x_1 \mathbf{y} + \eta \begin{pmatrix}
		0 & x_0/\alpha \\
		x_0^{\sigma^2} & -(\beta/\alpha x_0)^\sigma
	\end{pmatrix}\mathbf{y}^\sigma + \begin{pmatrix}
		-\beta x_0/\alpha & -(x_0/\alpha)^{\overline{\sigma}} \\
		-x_0^\sigma &0
	\end{pmatrix} \mathbf{y}^{\overline{\sigma}} \\
	&= x_1 \mathbf{y}+\eta T_{x_0}(\mathbf{y})+\det(M_{T_{x_0}})T_{x_0}^{-1}(\mathbf{y}),
\end{align*}
where $T_{x_0}\in \Gamma L (2,L)$ for $x_0 \neq 0$ has associated field automorphism $\sigma$ and 
\[M_{T_{x_0}} = \begin{pmatrix}
		0 & x_0/\alpha \\
		x_0^{\sigma^2} & -(\beta/\alpha x_0)^\sigma
	\end{pmatrix}.\]
	Note that $X^{\sigma+1}-\beta X-\alpha \in L[X]$ has no roots in $L$ if and only if $X^{\sigma+1}-(\beta/\alpha)^\sigma X-1/\alpha \in L[X]$ has no roots, which can easily be verified by multiplying the second polynomial by $\alpha^{\sigma+1}$.	We conclude that the Taniguchi semifields are, via isotopy and the Knuth orbit, equivalent to Construction~\ref{thm:general_structure} with the admissible mapping from Proposition~\ref{prop:admissible2}. Note that the Taniguchi semifields yield (up until now) asymptotically the largest known family of semifields in odd characteristic~\cite{golouglu2023counting}. This in particular shows that Construction~\ref{thm:general_structure} is so far the most general construction known --- covering both the largest known constructions (the Taniguchi semifields) as well as other, non-isotopic, semifields like the commutative Zhou-Pott semifields.
	
\section{Investigating new semifields produced by the construction} \label{sec:new}

In this section, we consider the semifields that are produced by Construction~\ref{thm:general_structure} with the admissible mapping from Proposition~\ref{prop:admissible1}. As we have seen in the previous section, for the choice $\eta=-1$, we recover exactly (up to equivalence via Knuth orbit) the Zhou-Pott semifields. It is however clear that the construction also yields other semifields, for instance in characteristic 2. The semifield multiplication we investigate is thus
\begin{align} \label{eq:investigate}
	\begin{pmatrix}x_0 \\ x_1 \end{pmatrix} \circ \begin{pmatrix}y_0 \\ y_1 \end{pmatrix} &=y_0 \mathbf{x} + \eta \begin{pmatrix}
		0 &  y_1\alpha \\
		y_1^\tau & 0
	\end{pmatrix}\mathbf{x}^\sigma + \begin{pmatrix}
		0 & -(y_1\alpha)^{\overline{\sigma}}\\
		-y_1^{\tau\overline{\sigma}} &0
	\end{pmatrix} \mathbf{x}^{\overline{\sigma}} \\
	&=\begin{pmatrix} x_0y_0+\eta \alpha x_1^\sigma y_1-x_1^{\overline{\sigma}}(\alpha y_1)^{\overline{\sigma}} \\
	x_1y_0+\eta x_0^\sigma y_1^\tau-x_0^{\overline{\sigma}}y_1^{\tau \overline{\sigma}}\end{pmatrix}. \nonumber
\end{align}
Firstly, let us note that we may disregard the case $\sigma^2=\id$:

\begin{proposition} \label{prop:halbecase}
	A semifield constructed with Construction~\ref{thm:general_structure} via an admissible mapping from Proposition~\ref{prop:admissible1} with field automorphism $\sigma$ such that $\sigma^2=\id$  is isotopic to a semifield constructed with Construction~\ref{thm:general_structure_2} with the same admissible mapping. In particular, it is isotopic to a generalized Dickson semifield.
\end{proposition}
\begin{proof}
	We have for $\sigma^2=\id$ (see Eq.~\eqref{eq:investigate})
	\begin{equation}
	\begin{pmatrix}x_0 \\ x_1 \end{pmatrix} \circ \begin{pmatrix}y_0 \\ y_1 \end{pmatrix} =y_0 \mathbf{x} + \begin{pmatrix}
		0 &  P(-y_1\alpha) \\
		P(-y_1^\tau)& 0
	\end{pmatrix}\mathbf{x}^\sigma.
	\label{eq:poop}
	\end{equation}
	where $P(y)=y^\sigma-\eta y$. We prove that $P$ (as a mapping on $L$) permutes $L$. This is clear if $\sigma=\id$ (recall $\eta\neq 1$ by the condition on $\eta$). So assume that $\sigma\neq \id$ and let $K$ be the fixed field of $\sigma$, i.e. $[L\colon K]=2$. Note that $P$ is additive, so it suffices to show that $P(y)=0$ has only the solution $y=0$ in $L$. So let $y^\sigma-\eta y=0$ for $y\neq 0$, i.e. $y^{\sigma-1}=\eta$. Then $\eta^{\sigma+1}=y^{\sigma^2-1}=1$ since $\sigma^2=\id$. But we have $\eta^{\sigma+1}=N_{L\colon K}(\eta)\neq 1$ by the condition in  Construction~\ref{thm:general_structure}, yielding a contradiction. We conclude that $P$ permutes $L$. We can thus perform an additive transformation (preserving isotopy) by applying $\left(\begin{smallmatrix} P^{-1} & 0 \\ 0 & P^{-1} \end{smallmatrix}\right)$ to Eq.~\eqref{eq:poop}, which results in
		\[\begin{pmatrix}
			P^{-1} & 0 \\
			0 & P^{-1}
		\end{pmatrix} y_0 \mathbf{x} + \begin{pmatrix}
		0 &  -y_1\alpha \\
		-y_1^\tau& 0
	\end{pmatrix}\mathbf{x}^\sigma,\]
	and a further transformation $y_0 \mapsto P(y_0)$ yields the desired equivalence to a semifield constructed with Construction~\ref{thm:general_structure_2}. As discussed in the previous section, these are generalized Dickson semifields (see also Table~\ref{t:1}).
\end{proof}

We will now perform some transformations on Eq.~\eqref{eq:investigate} that simplify our analysis but preserve isotopy.
Firstly, we apply an transformation $y_1 \mapsto -y_1/\alpha$, as well as renaming $\alpha \mapsto 1/\alpha$ that will make this semifield easier to analyze and compare to the existing semifields, in particular the Zhou-Pott semifield. We arrive at:

\[	\begin{pmatrix}x_0 \\ x_1 \end{pmatrix} \circ \begin{pmatrix}y_0 \\ y_1 \end{pmatrix}=\begin{pmatrix} x_0y_0-\eta x_1^\sigma y_1+x_1^{\overline{\sigma}}y_1^{\overline{\sigma}} \\
	x_1y_0-\eta x_0^\sigma (\alpha y_1)^\tau+x_0^{\overline{\sigma}}(\alpha y_1)^{\tau \overline{\sigma}}\end{pmatrix}.\]
	
	We now compute its transpose:
	
	\[	\begin{pmatrix}x_0 \\ x_1 \end{pmatrix} \circ^t \begin{pmatrix}y_0 \\ y_1 \end{pmatrix}=\begin{pmatrix} x_1y_1^\sigma-\eta x_1^\sigma y_1+\alpha(x_0y_0^\sigma-\eta x_0^\sigma y_0)^{\overline{\tau}} \\
	x_0y_1+x_1y_0\end{pmatrix}.\]
	Here, we are easily able to discern again the commutative Zhou-Pott semifields for $\eta=-1$. We finish our preparation by applying $\tau$ to $x_0,y_0$ (clearly preserving isotopy again):
		\begin{equation}
					\begin{pmatrix}x_0 \\ x_1 \end{pmatrix} \circ \begin{pmatrix}y_0 \\ y_1 \end{pmatrix}=\begin{pmatrix} x_1y_1^\sigma-\eta x_1^\sigma y_1+\alpha(x_0y_0^\sigma-\eta x_0^\sigma y_0) \\
	x_0^\tau y_1+x_1y_0^\tau \end{pmatrix}.
		\label{eq:ZP_final}
		\end{equation}
	The reason we choose to work with this representation of the semifields is that both components of the semifield are homogeneous polynomials (recall that $\sigma$ and $\tau$ are both Frobenius automorphisms, i.e., mappings of the form $x \mapsto x^{p^k}$). We call semifields with this property \emph{biprojective}, and strong tools to investigate semifields with this property have been developed by G\"olo\u{g}lu and the author~\cite{golouglu2023counting,golouglu2023exponential}:
	
	\begin{definition}
		Let $\S$ be a semifield  with semifield multiplication
		\[	\begin{pmatrix}x_0 \\ x_1 \end{pmatrix} \circ \begin{pmatrix}y_0 \\ y_1 \end{pmatrix} = \begin{pmatrix} f(x_0,x_1,y_0,y_2) \\ g(x_0,x_1,y_0,y_2) \end{pmatrix}\]
		defined on $L^2$. We call $\S$ a $(\sigma,\tau)$-{biprojective} semifield if $\sigma \colon x \mapsto x^{p^k}$, $\tau \colon x \mapsto x^{p^l}$, $f$ is $(p^k+1)$-homogeneous and $g$ is $(p^l+1)$-homogeneous.
	\end{definition}
	
	Many known semifields (including most of the semifields we investigated in this paper) are (equivalent to) biprojective semifields: Dickson (automorphism pair $(\id,\sigma)$), Knuth I and Taniguchi (both $\sigma,\sigma^2$), Knuth II/Hughes-Kleinfeld ($\sigma,\sigma$), Bierbrauer and Dempwolff (both $\sigma,\id$) and Zhou-Pott ($\sigma,\tau$), the necessary transformations to bring these semifields into the biprojective form are usually easy to spot, some are displayed in~\cite{golouglu2023counting,golouglu2023exponential}.\\

	For the rest of this section, we will denote by $\S_{\sigma,\tau,\alpha,\eta}$ the semifield defined on $L^2$ with multiplication as defined in Eq.~\eqref{eq:ZP_final}. \\

	We want to investigate the  four following natural questions:
	\begin{enumerate}
		\item For which parameters is $\S_{\sigma,\tau,\alpha,\eta}$ isotopic to $\S_{\sigma',\tau',\alpha',\eta'}$?
		\item How many non-isotopic semifields of size $p^{2m}$ are defined by this construction?
			\item What are the nuclei of the semifields $\S_{\sigma,\tau,\alpha,\eta}$?
		\item Can we prove that this construction produces semifields that are non-isotopic to any known construction so far?
	\end{enumerate}
		We will answer all these questions in this section. In particular, we show that our family contains many semifields non-isotopic to the Zhou-Pott family it contains as well as semifields that are inequivalent to any other known semifield.

		\subsection{Isotopies inside the family $\S_{\sigma,\tau,\alpha,\eta}$}
		
	We start with observing simple isotopies inside the family of semifields $\S_{\sigma,\tau,\alpha,\eta}$.
	
	\begin{proposition} \label{prop:isotopies}
	\begin{enumerate}
	The semifield $\S_{\sigma,\tau,\alpha,\eta}$ is isotopic to $\S_{\sigma,\overline{\tau},1/\alpha,\eta}$ and $\S_{\overline{\sigma},\tau,\alpha,1/\eta}$.
	\end{enumerate}

	\end{proposition}
	\begin{proof}
		We start with the first claim. We perform the following isotopic transformations: Starting with the multiplication of $\S_{\sigma,\tau,\alpha,\eta}$ in Eq.~\eqref{eq:ZP_final}, we exchange $x_0 \leftrightarrow x_1$ and $y_0 \leftrightarrow y_1$, arriving at
				\begin{equation*}
				\begin{pmatrix}x_0 \\ x_1 \end{pmatrix} \circ' \begin{pmatrix}y_0 \\ y_1 \end{pmatrix} =\begin{pmatrix} x_0y_0^\sigma-\eta x_0^\sigma y_0+\alpha(x_1y_1^\sigma-\eta x_1^\sigma y_1) \\
	x_1^\tau y_0+x_0y_1^\tau \end{pmatrix}.
		\end{equation*}
		We then divide the first component by $\alpha$ and apply $\overline{\tau}$ to the second component, arriving at the multiplication of $\S_{\sigma,\overline{\tau},1/\alpha,\eta}$.
		
		The second claim follows similarly, the transformations applied to $\S_{\sigma,{\tau},\alpha,\eta}$ are: Applying $\overline{\sigma}$ to $x_0,x_1,y_0,y_1$ and then dividing the first component by $-1/\eta$ and applying $\sigma$ to the second component.
		
	\end{proof}
	
	We now use the group theoretic machinery derived for determining if biprojective semifields are isotopic are not. The tools were developed in~\cite{golouglu2023exponential} for commutative biprojective semifields and generalized to non-commutative biprojective semifields in~\cite{golouglu2023counting}. 
	
	We introduce the necessary notation: Recall that we denote by $\Aut(\S) \leq \GL(\F_p^n)^3$ the autotopism group of the semifield $\S$ with $p^n$ elements, i.e. the set of all isotopisms from $\S$ to itself. It is well known and easy to see that isotopic semifields have conjugate autotopism groups (see e.g.~\cite[Lemma 5.1.]{golouglu2023exponential}).
	
	We write mappings $A \in \End_{\F_p}(L^2)$ as $2 \times 2$ matrices
of       $\F_{p}$-linear mappings     from $L$ to itself.
That is,
\[
A=\begin{pmatrix}
	A_1 & A_2 \\
	A_3 & A_4
\end{pmatrix}, \textrm{ for } A_i \in \End_{\F_p}(L). 
\]
We call the constituent functions $A_1,\dots,A_4$ of $A$ \textit{subfunctions of} $A$.
Set for two field automorphisms $\sigma,\tau \in \Gal(L)$
\[
\gamma_r=(N_r,L_r,M_r)\in \GL(L)^3  \textrm{ with } 
N^{(\sigma,\tau)}_r=\begin{pmatrix}
	m_{r^{\sigma+1}} & 0 \\
	0 & m_{r^{\tau+1}}
\end{pmatrix}, \quad 
L_r=M_r=\begin{pmatrix}
	m_{r} & 0 \\
	0 & m_{r}
\end{pmatrix}, 
\]
where $m_r$ denotes  multiplication with the finite field element $r \in L^*$. 
For simplicity, we write these diagonal matrices also in the form $\diag(m_{r},m_{r})$, 
so 
\[
\gamma^{(\sigma,\tau)}_r = (\diag(m_{r^{\sigma+1}},m_{r^{\tau+1}}),\diag(m_{r},m_{r}),\diag(m_{r},m_{r})).
\]

The crucial fact is that
$\gamma^{(\sigma,\tau)}_r \in \Aut(\S)$ for all $r \in L^*$ when $\S$ is a $(\sigma,\tau)$-biprojective semifield, which follows immediately from the definition of biprojectivity:

\begin{lemma} \label{lem_first}
Let $\S$ be a $(\sigma,\tau)$-biprojective semifield on $L^2$. Then 
$\gamma^{(\sigma,\tau)}_r\in \Aut(\S)$ for all $r \in L^*$.
\end{lemma}

We fix some further notation that we will use for the remainder of this section (on top of the notation already introduced in Notation~\ref{not}):
\begin{notation}
\begin{itemize}
\item[]
\item Let $\S$ be a $(\sigma,\tau)$-biprojective semifield.
\item $\GL(L^2) \cong \GL(2m,p)$ is the set of all invertible linear mappings of $L^2$ as an $\F_p$-vector space.
\item Define the cyclic group 
\[
	Z^{(\sigma,\tau)} = \{ \gamma^{(\sigma,\tau)}_r : r \in L^*\} \leq \GL(L^2)^3
\] of order $p^m-1$. By Lemma~\ref{lem_first}, we have $Z^{(\sigma,\tau)}\leq \Aut(\S)$.
\item Let $p'$ be a $p$-primitive divisor of $p^m-1$, i.e. $p'|p^m-1$ and $p'\nmid p^{t}-1$ for $t<m$. Such a prime $p'$ always exists if 
$m > 2$ and $(p,m) \neq (2,6)$ by Zsigmondy's Theorem (see e.g.~\cite[Chapter IX., Theorem 8.3.]{HuppertII}).  
\item Let $R$ be the unique Sylow $p'$-subgroup of $L^*$.
\item Define
\[
Z_R^{(\sigma,\tau)} = \{ \gamma_r \colon r \in R\},
\]
which is the unique Sylow $p'$-subgroup of $Z^{(\sigma,\tau)}$ with $|R|$ elements. 
\item For a $(\sigma,\tau)$-biprojective semifield $\S$, denote by 
\[
C_{\S} = C_{\Aut(\S)}(Z_R^{(\sigma,\tau)}),
\]
the centralizer of $Z_R^{(\sigma,\tau)}$ in $\Aut(\S)$.
\end{itemize}
\end{notation}

The basic idea to prove that two biprojective semifields are not isotopic is to show that if they are isotopic, then the respective subgroups $Z_R^{(\sigma,\tau)}$ have to be conjugate, which imposes strong conditions on potential isotopisms. The proof uses sophisticated tools from group theory, see \cite{golouglu2023exponential}. We cite the end result which we are going to use, it was proven by G\"olo\u{g}lu and the author in \cite[Theorem 5.10.]{golouglu2023exponential} for commutative biprojective semifields and generalized to non-commutative biprojective semifields in~\cite[Theorem 3]{golouglu2023counting} (in fact, it was proven specifically for the biprojective Taniguchi semifields but holds in the general case as well). 
\begin{theorem} \label{thm:biproj}
	Let $L=\F_{p^m}$ with $m>2$, $(p,m)\neq (2,6)$, and $\S_1$, $\S_2$ be $(\sigma_1,\tau_1$)- and $(\sigma_2,\tau_2)$-biprojective semifields, respectively, defined on $L^2$. Assume that $\tau_1^2,\sigma_1^2 \neq \id$, $\sigma_1 \notin \{\tau_1,\overline{\tau_1}\}$ and
	\begin{equation}
		C_{\S_1} \text{ contains }Z^{(\sigma,\tau)} \text{ as an index } I \text{ subgroup such that }p'\text{ does not divide }I.
	\tag{C}\label{eq:condition}
	\end{equation}
	If $\S_1$ and $\S_2$ are isotopic, then there exists an isotopism $\gamma=(N_1,N_2,N_3) \in \Gamma L(L^2)^3$ such that $N_2,N_3 \in \Gamma L (2,L)$. Furthermore, 
$\sigma_2,\tau_2 \in \{\sigma_1,\overline{\sigma_1},\tau_1,\overline{\tau_1}\}$ and if $\sigma_1=\sigma_2$ and $\tau_1=\tau_2$ then $\gamma \in \Gamma L (2,L)^3$, $N_1$ is a diagonal matrix, and the field automorphisms associated with $N_1,N_2,N_3$ are the same. Similarly, if $\sigma_1=\tau_2$ and $\tau_1=\sigma_2$ then $N_1$ is anti-diagonal, the field automorphisms of $N_1,N_2,N_3$ coincide.
\end{theorem}

We start by proving Condition~\eqref{eq:condition} for the semifields $\S_{\sigma,\tau,\alpha,\eta}$.

\begin{lemma} \label{lem:condition}
	Let $\S=\S_{\sigma,\tau,\alpha,\eta}$ be the semifield with multiplication $\circ$ and $\sigma \notin \{\tau,\overline{\tau}\}$ and $\sigma^2 \neq \id$. Then 
	$|C_{\S}| =(p^m-1)\cdot \gcd(p^k+1,p^m-1) \cdot (p^{\gcd(k,m)}-1)$.
	In particular, Condition~\eqref{eq:condition} is satisfied.
\end{lemma}
\begin{proof}
	Let $\gamma=(N_1,N_2,N_3) \in C_{\S_1}$.
	By~\cite[Lemma 5.7.]{golouglu2023exponential}, $N_2,N_3 \in \GL(2,L)$.  So let $N_i=\left(\begin{smallmatrix} a_i & b_i \\ c_i & d_i\end{smallmatrix} \right)$ for $i \in \{2,3\}$. The second component of the equation $N_1(\left(\begin{smallmatrix} x_0 \\ x_1\end{smallmatrix} \right)\circ \left(\begin{smallmatrix} y_0 \\ y_1\end{smallmatrix} \right))=N_2\left(\begin{smallmatrix} x_0 \\ x_1\end{smallmatrix} \right) \circ N_3\left(\begin{smallmatrix} y_0 \\ y_1\end{smallmatrix} \right)$ is then
		\begin{align*}
		A_3(x_1y_1^\sigma-\eta x_1^\sigma y_1+&\alpha(x_0y_0^\sigma-\eta x_0^\sigma y_0))+ A_4(x_0^\tau y_1+x_1y_0^\tau) \\&=(a_2x_0+b_2x_1)^\tau(c_3y_0+d_3y_1) + (c_2x_0+d_2x_1)(a_3y_0+b_3y_1)^\tau,
		\end{align*}
		where $A_3,A_4 \in \End_{\F_p}(L)$. By comparing coefficients, it is immediately clear that $A_3=0$ and $A_4(x)=d_1x$ for some $d_1 \in L^*$. The equation is then satisfied if and only if the following conditions are satisfied (comparing coefficients for all monomials):
	\begin{equation} \label{eq:conditions_1}
	d_1=a_2^\tau d_3=d_2a_3^\tau \text{ and } 0=a_2c_3=b_2c_3=b_2d_3=c_2a_3=d_2b_3,
	\end{equation}
	implying $b_2=b_3=c_2=c_3=0$. We then consider the first component similarly, getting:
			\begin{align*}
		A_1(x_1y_1^\sigma-\eta x_1^\sigma y_1+&\alpha(x_0y_0^\sigma-\eta x_0^\sigma y_0))+ A_2(x_0^\tau y_1+x_1y_0^\tau) \\&= d_2 d_3^\sigma x_1 y_1^\sigma-\eta d_2^\sigma d_3 x_1^\sigma y_1 + \alpha (a_2a_3^\sigma x_0y_0^\sigma - \eta a_2^\sigma a_3 x_0^\sigma y_0),
		\end{align*}
		where $A_1,A_2 \in \End_{\F_p}(L)$.
		Similarly, this implies $A_2=0$, $A_1(x)=a_1x$ for $a_1 \in L^*$ and 
		\[a_1=a_2^\sigma a_3=a_2a_3^\sigma=d_2d_3^\sigma=d_2^\sigma d_3,\]
		in particular $a_2^{\sigma-1}=a_3^{\sigma-1}$ and $d_2^{\sigma-1}=d_3^{\sigma-1}$. So write $a_3=a_2 \zeta_1$, $d_3=d_2 \zeta_2$ where $\zeta_1,\zeta_2 \in K^*$ and $K$ is the fixed field of $\sigma$. Plugging this into Eq.~\eqref{eq:conditions_1} yields $\zeta_2=\zeta_1^\tau$. 
		$a_2a_3^\sigma=d_2d_3^\sigma$ then implies $(a_2/d_2)^{\sigma+1}=\zeta_1^{\tau-1}$. Note that for any $\zeta_1$ we can find $a_2,d_2$ satisfying this equation (since $x\mapsto x^{\sigma+1}$ is just the square-function on $K$, so we can choose $(a_2/d_2)=\zeta_1^{(p^l-1)/2}$ if $p$ is odd, for even $p$ the function $x\mapsto x^{\sigma+1}$ is even bijective on $K$ and allows a solution immediately). We thus have for fixed $\zeta_1$ exactly $(p^m-1)\cdot \gcd(p^k+1,p^m-1)$ many solutions, since we can choose $a_2 \in L^*$ arbitrarily and then $d_2$ uniquely up to a $(p^k+1)$-st power.
		
		We conclude that $|C_{\S_1}|=(p^m-1)\cdot \gcd(p^k+1,p^m-1) \cdot (p^{\gcd(k,m)}-1)$. Note that $p'$ does not divide $\gcd(p^k+1,p^m-1)$ (since otherwise it would divide $p^{2k}-1$ which is impossible since $p'$ is a $p$-primitive divisor of $p^m-1$ and $2k\neq m$). We conclude that  Condition~\eqref{eq:condition} is satisfied.
\end{proof}

We can thus apply Theorem~\ref{thm:biproj} to the semifields $\S_{\sigma,\tau,\alpha,\eta}$. By Propositions~\ref{prop:halbecase} and \ref{prop:isotopies}, we may assume without loss of generality that $k<m/2$, $l\leq m/2$. If we exclude the special cases $2l=m$ and $k=l$, Theorem~\ref{thm:biproj} then implies that $\S_{\sigma_1,\tau_1,\alpha_1,\eta_1}$ can only be isotopic to $\S_{\sigma_2,\tau_2,\alpha_2,\eta_2}$ if $\sigma_1=\sigma_2$ and $\tau_1=\tau_2$ or $\sigma_1=\tau_2$ and $\tau_1=\sigma_2$. We first show that the second case cannot occur:
\begin{proposition} \label{prop:antidiag}
	Let $\S_1=(L^2,+,\circ_1)=\S_{\sigma,\tau,\alpha_1,\eta_1}$ and $\S_2=(L^2,+,\circ_2)=\S_{\tau,\sigma,\alpha_2,\eta_2}$ be two semifields defined on $L^2$ with $\sigma \colon x \mapsto x^{p^k}$, $\tau \colon x \mapsto x^{p^l}$, $k,l<m/2$ and $k\neq l$. Then $\S_1$ and $\S_2$ are not isotopic.
\end{proposition}
\begin{proof}
		Assume we have an isotopism $(N_1,N_2,N_3)$, i.e.
	\begin{equation}
		N_1(\left(\begin{smallmatrix}x_0 \\x_1\end{smallmatrix}\right)\circ_1 \left(\begin{smallmatrix}y_0 \\y_1\end{smallmatrix}\right))=N_2\left(\begin{smallmatrix}x_0 \\x_1\end{smallmatrix}\right) \circ_2 N_3\left(\begin{smallmatrix}y_0 \\y_1\end{smallmatrix}\right).
	\label{eq:general2}
	\end{equation}
	By Theorem~\ref{thm:biproj}, $\gamma=(N_1,N_2,N_3) \in \Gamma L (2,L)$, $N_1$ is an anti-diagonal matrix and $N_1,N_2,N_3$ have the same associated field automorphism $\rho$. So let $N_i=\left(\begin{smallmatrix} a_i & b_i \\ c_i & d_i\end{smallmatrix} \right)$ for $i \in \{1,2,3\}$ where $a_1=d_1=0$. The second component of Eq.~\eqref{eq:general2} is:
			\begin{align*}
		c_1\left(x_1y_1^\sigma-\eta_1 x_1^\sigma y_1+\alpha_1(x_0y_0^\sigma-\eta_1 x_0^\sigma y_0)\right)^\rho =((a_2^{\overline{\rho}}x_0&+b_2^{\overline{\rho}}x_1)^{\tau}(c_3^{\overline{\rho}}y_0+d_3^{\overline{\rho}}y_1) \\&+ (c_2^{\overline{\rho}}x_0+d_2^{\overline{\rho}}x_1)(a_3^{\overline{\rho}}y_0+b_3^{\overline{\rho}}y_1)^{\tau})^\rho.
		\end{align*}
	By comparing the degrees of the monomials on both sides of the equation (recalling that $\tau \notin \{\sigma,\overline{\sigma}\}$), it is immediate that no isotopism exists.
\end{proof}

With this case out of the way, we can finish the complete result.

\begin{theorem} \label{thm:isotop}
	Let $\S_1=(L^2,+,\circ_1)=\S_{\sigma,\tau,\alpha_1,\eta_1}$ and $\S_2=(L^2,+,\circ_2)=\S_{\sigma_2,\tau_2,\alpha_2,\eta_2}$ be two semifields defined on $L^2$ with $\sigma \colon x \mapsto x^{p^k}$, $\tau \colon x \mapsto x^{p^l}$, $k,l<m/2$ and $k\neq l$. Let $K$ be the fixed field of $\sigma$. $\S_1$ and $\S_2$ are isotopic if and only if $\sigma_2=\sigma$, $\tau_2=\tau$, and there exists a field automorphism $\rho$ of $L$ such that
	\begin{itemize}
		\item $N_{L\colon K}(\eta_1)^\rho=N_{L\colon K}(\eta_2)$, and 
		\item $\frac{\alpha_1^{\rho}}{\alpha_2} \in L^{\sigma+1}L^{\tau-1}$. 
	\end{itemize}
\end{theorem}
\begin{proof}
With Proposition~\ref{prop:antidiag} it remains to deal with the question $\sigma_2=\sigma$, $\tau_2=\tau$. 
		The proof is very similar to the one in Lemma~\ref{lem:condition}. 
We consider when 
	\begin{equation}
		N_1(\left(\begin{smallmatrix}x_0 \\x_1\end{smallmatrix}\right)\circ_1 \left(\begin{smallmatrix}y_0 \\y_1\end{smallmatrix}\right))=N_2\left(\begin{smallmatrix}x_0 \\x_1\end{smallmatrix}\right) \circ_2 N_3\left(\begin{smallmatrix}y_0 \\y_1\end{smallmatrix}\right).
	\label{eq:general}
	\end{equation}
	By Theorem~\ref{thm:biproj}, $\gamma=(N_1,N_2,N_3) \in \Gamma L (2,L)$, $N_1$ is a diagonal matrix and $N_1,N_2,N_3$ have the same associated field automorphism $\rho$. So let $N_i=\left(\begin{smallmatrix} a_i & b_i \\ c_i & d_i\end{smallmatrix} \right)$ for $i \in \{1,2,3\}$ where $b_1=c_1=0$. The second component of Eq.~\eqref{eq:general} is:
			\begin{equation*}
		d_1\left(x_0^\tau y_1+x_1y_0^\tau\right)^\rho =\left((a_2^{\overline{\rho}}x_0+b_2^{\overline{\rho}}x_1)^{\tau}(c_3^{\overline{\rho}}y_0+d_3^{\overline{\rho}}y_1) + (c_2^{\overline{\rho}}x_0+d_2^{\overline{\rho}}x_1)(a_3^{\overline{\rho}}y_0+b_3^{\overline{\rho}}y_1)^{\tau}\right)^\rho.
		\end{equation*}
		Comparing coefficients leads to: $b_2=b_3=c_2=c_3=0$ and $d_1^\rho=a_2^\tau d_3=d_2a_3^\tau$. The first component of Eq.~\eqref{eq:general} then is:
		\begin{align*}
	a_1(x_1y_1^\sigma-\eta_1 x_1^\sigma y_1+&\alpha_1(x_0y_0^\sigma-\eta_1 x_0^\sigma y_0))^\rho \\& = \left(\left(d_2 d_3^\sigma\right)^{\overline{\rho}}x_1 y_1^\sigma-\eta_2^{\overline{\rho}} \left(d_2^\sigma d_3\right)^{\overline{\rho}} x_1^\sigma y_1 + \alpha_2^{\overline{\rho}}(\left(a_2a_3^\sigma\right)^{\overline{\rho}} x_0y_0^\sigma - \eta_2^{\overline{\rho}} \left(a_2^\sigma a_3\right)^{\overline{\rho}} x_0^\sigma y_0)\right)^{\rho}.
		\end{align*}
		We derive the equations:
		\begin{equation}
				a_1^\rho = d_2d_3^\sigma, \;\; a_1^\rho\frac{\eta_1^\rho}{\eta_2} = d_2^\sigma d_3,\;\; a_1^\rho\frac{\alpha_1^\rho}{\alpha_2} = a_2 a_3^\sigma,\;\;a_1^\rho\frac{(\alpha_1\eta_1)^\rho}{\alpha_2\eta_2} = a_2^\sigma a_3.
		\label{eq:conditions}
		\end{equation}

		Comparing the first two equations, and the last two equations in \eqref{eq:conditions} leads to 
		\[\frac{\eta_1^\rho}{\eta_2} = \left(\frac{d_2}{d_3}\right)^{\sigma-1}=\left(\frac{a_2}{a_3}\right)^{\sigma-1}.\]
		Note that the $(\sigma-1)$-st powers are exactly the elements $x \in L$ such that $N_{L\colon K}(x)=1$. This implies the condition $N_{L\colon K}(\eta_1)^\rho=N_{L\colon K}(\eta_2)$. Let $z=a_2/a_3$, then $d_2/d_3=z\zeta$ where $\zeta \in K$. Comparing the first and the third equation in \eqref{eq:conditions} finally gives
		\[z\zeta d_3^{\sigma+1}\frac{\alpha_1^\rho}{\alpha_2}=za_3^{\sigma+1}.\]
		From the second component, we had $a_2^\tau d_3=d_2a_3^\tau$, which is equivalent to $\zeta=z^{\tau-1}$. The previous condition then turns into 
		\[z^{\tau-1}\left(\frac{d_3}{a_3}\right)^{\sigma+1}\frac{\alpha_1^\rho}{\alpha_2}=1.  \]
		We can find $z,a_3,d_3$ satisfying this condition if and only if $\frac{\alpha_1^\rho}{\alpha_2} \in L^{\sigma+1}L^{\tau-1}$. With this, all conditions are satisfied and the theorem is proven.
\end{proof}

\begin{remark}
		If $m/\gcd(k,m)$ is odd, then the set of $(\sigma+1)$-st powers is exactly the set of squares by Lemma~\ref{lem:gcd} and the second condition of Theorem~\ref{thm:isotop} is always satisfied. This recovers the isotopy result of the Zhou-Pott semifields in~\cite{ZP13}. Indeed, the Zhou-Pott family of semifields contain only $O(m^2)$ non-isotopic semifields of order $p^{2m}$, since all choices of $\alpha$ yield isotopic semifields, so only different choices for $\sigma,\tau$ give non-isotopic semifields. 
\end{remark}

\begin{corollary} \label{cor:ZP}
	The semifield $\S_{\sigma,\tau,\alpha,\eta}$ is isotopic to a semifield from the Zhou-Pott family if and only if $N_{L\colon K}(\eta)=-1$, where $K$ is the fixed field of $\sigma$. 
\end{corollary}
\begin{proof}
	Follows from Theorem~\ref{thm:isotop}, recalling that the Zhou-Pott semifields are isotopic to a semifield of the form $\S_{\sigma,\tau,\alpha,-1}$. 
\end{proof}

\begin{corollary} \label{cor:count}
  For fixed $L=\F_{p^m}$ and automorphisms $\sigma,\tau \in \Gal(L)$ let $N(m,\sigma,\tau)$ be the number of non-isotopic semifields in the family $\S_{\sigma,\tau,\alpha,\eta}$. Then
	\[(1/m)(|K|-2)(d-1) \leq N(m,\sigma,\tau) \leq (|K|-2)(d-1).\]
	Here $K$ denotes the fixed field of $\sigma$ and $d=\gcd(p^k+1,p^l-1,p^m-1)$.
\end{corollary}
\begin{proof}
	This follows again immediately from Theorem~\ref{thm:isotop}: We can choose $N_{L\colon K}(\eta)$ freely in $K\setminus\{0,1\}$, yielding $|K|-2$ choices. Similarly, there are $\gcd(p^k+1,p^l-1,p^m-1)-1$ non-trivial cosets of $L^{\sigma+1}L^{\tau-1}$ (see the proof of Proposition~\ref{prop:admissible1}). The factor $1/m$ is derived from the number of choices for the field automorphism $\rho$ in Theorem~\ref{thm:isotop}.
\end{proof}
 If $m$ is divisible by three, we can choose $\sigma$ as an order $3$ automorphism of $L$, so $K$ has order $p^{m/3}$, and we get at least $\approx p^{m/3}$ many inequivalent semifields.  In particular, the number of non-isotopic semifields of this form of order $p^{2m}$ is not bounded above by a polynomial in $m$.

This in particular shows that the Zhou-Pott semifields make up only a tiny fraction of the semifields $\S_{\sigma,\tau,\alpha,\eta}$. We want to remark that the value of $d=\gcd(p^k+1,p^l-1,p^m-1)$ in Corollary~\ref{cor:count} can be precisely evaluated using Lemma~\ref{lem:gcd}, we chose to keep this form to avoid technical case distinctions.

Our next objective is to show that the $\S_{\sigma,\tau,\alpha,\eta}$ family contains semifields not contained in \emph{any} other known family of semifields. Recall that the nuclei of a semifield are an equivalence invariant, so the first step is to compute the nuclei of our semifields $\S_{\sigma,\tau,\alpha,\eta}$.

\subsection{The nuclei of $\S_{\sigma,\tau,\alpha,\eta}$}
To compute the nuclei, we use the following technique developed in~\cite{MP12}. Denote by $\mathcal{C}$ the set of right-multiplications of a semifield $\S$ with multiplication $\circ$, i.e.
\[\mathcal{C} = \{R_y \colon x \mapsto x \circ y \mid y \in \S\} \subseteq \End_{\F_p}(\S).\]
This set is usually called the \textbf{spread set} of $\S$ because of its close link to the semifield spread $\Sigma$ (see Eq.~\eqref{eq:ABB}).
\begin{theorem}{\cite[Theorem 2.2.]{MP12}} \label{thm:nuclei}
	Let $\S$ be a semifield and $\mathcal{C}$ the corresponding set of right-multiplications and $\mathcal{C}^d$ the set of right-multiplications of the dual semifield $\pi_1(\S)$. 
	\begin{enumerate}
		\item The right nucleus $\mathcal{N}_r(\S)$ is isomorphic to the largest field $F$ contained in $\End_{\F_p}(\S)$ such that $F\mathcal{C}\subseteq \mathcal{C}$. 
		\item The middle nucleus $\mathcal{N}_m(\S)$ is isomorphic to the largest field $F$ contained in $\End_{\F_p}(\S)$ such that $\mathcal{C}F\subseteq \mathcal{C}$. 
		\item The left nucleus $\mathcal{N}_l(\S)$ is isomorphic to the largest field $F$ contained in $\End_{\F_p}(\S)$ such that $F\mathcal{C}^d\subseteq \mathcal{C}^d$. 
	\end{enumerate}
\end{theorem} \label{thm:det_nuclei}
With this result, determining the nuclei is reduced to a manageable (yet technical) calculation.
\begin{theorem}
	The nuclei of $\S_{\sigma,\tau,\alpha,\eta}$ are
	\begin{itemize}
		\item $\mathcal{N}_r(\S) \cong \mathcal{N}_l(\S) \cong \F_{p^{\gcd(k,l,m)}}$,
		\item $\mathcal{N}_m(\S) \cong \F_{p^{\gcd(2k,m)}}$.
	\end{itemize}
\end{theorem}
\begin{proof}
	We apply Theorem~\ref{thm:nuclei}, so let $\mathcal{C} \subseteq \End_{\F_p}(L^2)$ be the set of right-multiplications of $\S_{\sigma,\tau,\alpha,\eta}$. Denote the elements in $\mathcal{C}$ by $R_{y_0,y_1}=\left(\begin{smallmatrix} R_{y_0,y_1}^{(1)} \\R_{y_0,y_1}^{(2)}\end{smallmatrix}\right)$. 
We start with the right nucleus. Let $A \in \End(L^2)$, which we write $A=\left(\begin{smallmatrix} A_1 & A_2 \\A_3 & A_4 \end{smallmatrix}\right)$ with subfunctions $A_i \in \End_{\F_p}(L)$. We check when $A\mathcal{C} \subseteq \mathcal{C}$. This means that for each pair $(y_0,y_1) \in L^2$ there has to be a unique pair $(u_0,u_1)\in L^2$ such that
	\[A_1\circ R_{y_0,y_1}^{(1)}+A_2\circ R_{y_0,y_1}^{(2)}=R_{u_0,u_1}^{(1)}, \text{ and }A_3\circ R_{y_0,y_1}^{(1)}+A_4\circ R_{y_0,y_1}^{(2)}=R_{u_0,u_1}^{(2)}.\]
Let us start with the second condition, which is
\[A_3(x_1y_1^\sigma-\eta x_1^\sigma y_1+\alpha(x_0y_0^\sigma-\eta x_0^\sigma y_0))+ A_4(x_0^\tau y_1+x_1y_0^\tau) = x_0^\tau u_1+x_1u_0^\tau.\]
It is immediate from degree comparisons that $A_3=0$ and $A_4(x)=dx$ for some $d \in L$. The condition then simplifies to $dy_1=u_1$, $dy_0^\tau=u_0^\tau$, which we rewrite as $d=u_1/y_1=(u_0/y_0)^\tau$ if $y_0y_1\neq 0$. 

Let us now look at the first condition. Similar degree considerations immediately yield again $A_2=0$, $A_1=ax$ for $a \in L$ which gives the conditions
\[a=\frac{u_0}{y_0}=\left(\frac{u_0}{y_0}\right)^\sigma = \frac{u_1}{y_1}=\left(\frac{u_1}{y_1}\right)^\sigma. \]
In particular, $a=d \in \F_{p^{\gcd(k,m)}}$ (i.e. the fixed field of $\sigma)$. From the previous condition $d=u_1/y_1=(u_0/y_0)^\tau$ we see that $d$ is also in the fixed field of $\tau$, so $a=d \in \F_{p^{\gcd(k,l,m)}}$. So the mappings $A$ satisfying $A\mathcal{C} \subseteq \mathcal{C}$ are precisely $A=\left(\begin{smallmatrix} m_a & 0 \\0 & m_a \end{smallmatrix}\right)$ for $a \in \F_{p^{\gcd(k,l,m)}}$, proving the claim. \\

Let us continue with the middle nucleus, so we check $\mathcal{C}A \subseteq \mathcal{C}$, which gives the similar condition
	\[R_{y_0,y_1}\begin{pmatrix}
		A_1(x_0)+A_2(x_1) \\
		A_3(x_0)+A_4(x_1)
	\end{pmatrix}=R_{u_0,u_1}\begin{pmatrix}
		x_0 \\
		x_1
	\end{pmatrix}.\]
	Checking the second component gives
	\[\left(A_1(x_0)+A_2(x_1)\right)^\tau y_1+\left(A_3(x_0)+A_4(x_1)\right) y_0^\tau=x_0^\tau u_1+x_1u_0^\tau.\]
	It is then clear (for instance by setting $y_0=0$) that $A_2=0$ or $A_2(x)=ax^{\overline{\tau}}$. Let us consider the first component:
		\begin{align*}
		\left(A_3(x_0)+A_4(x_1)\right) &y_1^\sigma-\eta \left(A_3(x_0)+A_4(x_1)\right)^\sigma y_1 	
		+\alpha \left(\left(A_1(x_0)+A_2(x_1)\right) y_0^\sigma-\eta\left(A_1(x_0)+A_2(x_1)\right)^\sigma y_0\right) \\&=x_1u_1^\sigma-\eta x_1^\sigma u_1+\alpha (x_0u_0^\sigma-\eta x_0^\sigma u_0).
		\end{align*}
		If $A_2(x)=ax^{\overline{\tau}}$ for $a \neq 0$ then (by comparing degrees again) we must have $\{\overline{\tau},\overline{\tau}\sigma\}=\{\id,\sigma\}$, which is equivalent to $\tau=\id$ or $\tau=\sigma$ and $\sigma^2=\id$. But we exclude precisely the cases $\tau=\id$ (cf. discussion after Proposition~\ref{prop:admissible1}) and $\sigma^2=\id$ (Proposition~\ref{prop:halbecase}) in our treatment. We conclude that $A_2=0$. With the same argumentation, we can infer $A_3=0$. For $A_1,A_4$ we get (still just comparing degrees) that $A_1(x)=ax$, $A_4(x)=dx$ with the following conditions (the first four from the first component, the last two from the second):
		\[dy_1^\sigma=u_1^\sigma, d^\sigma y_1=u_1,ay_0^\sigma =u_0^\sigma, a^\sigma y_0=u_0,a^\tau y_1=u_1, dy_0^\tau=u_0^\tau.\]
		For $y_0y_1 \neq 0$ this gives 
		\begin{align*}
				a&=\left(\frac{u_0}{y_0}\right)^\sigma=\left(\frac{u_0}{y_0}\right)^{\overline{\sigma}}=\left(\frac{u_1}{y_1}\right)^\tau \\
				d&=\left(\frac{u_1}{y_1}\right)^\sigma=\left(\frac{u_1}{y_1}\right)^{\overline{\sigma}}=\left(\frac{u_0}{y_0}\right)^\tau. 
		\end{align*}
	This is equivalent to $a,d \in \F_{p^{\gcd(2k,m)}}$, $d=a^{\tau\sigma}$. In particular, the middle nucleus is isomorphic to $\F_{p^{\gcd(2k,m)}}$ as claimed.
	
	The left nucleus is computed the same as the right nucleus, except we switch $\S$ with $\pi_1(\S)$, which is equivalent to considering left-multiplications instead of right-multiplications. The calculation is almost identical to the one for the right nucleus; and the result is that the left nucleus is isomorphic to $\F_{p^{\gcd(k,l,m)}}$ as well.
	\end{proof}

\subsection{$\S_{\sigma,\tau,\alpha,\eta}$ contains new semifields}

Theorem~\ref{thm:det_nuclei} shows that if a semifield is equivalent via isotopy and Knuth orbit to a semifield $\S_{\sigma,\tau,\alpha,\eta}$ then it needs to have two nuclei with coinciding order while the third nucleus can potentially be much larger. It is however also possible for semifields $\S_{\sigma,\tau,\alpha,\eta}$ to have all nuclei to be arbitrarily small, even down to the prime field $\F_p$ (by choosing appropriate field automorphisms). 

This means that many of the known constructions  (e.g. Dickson semifields, the semifields in~\cite{CARDINALI2006940,ebert2009infinite}) cannot contain our family of semifields as they have at least one large nucleus. On the other hand, the Taniguchi semifields have three nuclei that are all isomorphic to each other~\cite[Propositions 3-5]{taniguchisf} and can thus also not contain our family. Similarly, if two nuclei of Albert semifields coincide, then they all do, which follows from the calculations in~\cite{Albert61} by Albert. The Zhou-Pott semifields are (properly) contained in our family of semifields as we have shown in Corollary~\ref{cor:ZP}. The twisted cyclic semifields of size $p^{2m}$ can (among other examples) not contain any semifields with three nuclei of size $(p,p^k,p)$ with $k>1$ (see~\cite[Corollary 2]{sheekey2020new}) and thus also do not contain our family. Other families (like the Ganley semifields~\cite{Ganley}) do not yield semifields of order $p^{2m}$ for all $p,m$ and can also be excluded; and others contain only a few inequivalent semifields of a given order which means they also cannot contain our family by our counting results in Corollary~\ref{cor:count}. We can thus conclude:

\begin{theorem}
	The family $\S_{\sigma,\tau,\alpha,\eta}$ contains semifields that are not equivalent  to any known semifields.
\end{theorem}

\section{Conclusion and open problems}
In this work, we unified many constructions of semifields of order $p^{2m}$ and provided a new infinite family of such semifields. Theorem~\ref{thm:sheekey} is one of the main ingredients in the proof of Constructions~\ref{thm:general_structure} and~\ref{thm:general_structure_2}. We, however, only use Theorem~\ref{thm:sheekey} for the parameter $d=2$. The natural questions is:
\begin{problem}
	Is it possible to find constructions of semifields similar to Constructions~\ref{thm:general_structure} and~\ref{thm:general_structure_2} where the right-multiplications are mappings as in Theorem~\ref{thm:sheekeysingular} for $d>2$?
\end{problem}
In the current literature there are only a few constructions of semifields known that produce semifields of order $p^{dm}$ for any $p,m$ and a fixed odd $d>2$; outside of the twisted cyclic fields there are to the knowledge of the author only the constructions~\cite{Bierbrauer10,ZKW}. This might indicate that such constructions are not as easy.

For the parameter $d=2$ it remains also open if there are more admissible mappings outside of the trivial one and the ones in Propositions~\ref{prop:admissible1} and \ref{prop:admissible2} that allow us to construct new, non-iosotopic semifields via out constructions. We conjecture that this is not the case.

We would also like to reiterate Open Problem~\ref{prob} in a more general way: We know that Constructions~\ref{thm:general_structure} and~\ref{thm:general_structure_2} can produce \emph{commutative} semifields (e.g. the Zhou-Pott semifields, the Dickson semifields and the Budaghyan-Helleseth semifields). Since isotopy does not preserve commutativity, we pose the following problem:
\begin{problem}
	Determine which semifields produced by Constructions~\ref{thm:general_structure} and~\ref{thm:general_structure_2} are equivalent to commutative semifields.
\end{problem}
We want to note that a criterion to check when a semifield is isotopic to a commutative one was found by Ganley~\cite[Theorem 4]{ganley1972polarities}, it is however not straightforward to check.

The next open problems concern the connection between semifields and coding theory and finite geometry, respectively. It is possible to embed the twisted cyclic semifields in a family of maximum rank distance (MRD) codes, see~\cite{sheekey2020new}. Here the question is:
\begin{problem}
	Is it possible to find a construction of MRD codes that contains Constructions~\ref{thm:general_structure} and~\ref{thm:general_structure_2}?
\end{problem}

Constructions~\ref{thm:general_structure} and~\ref{thm:general_structure_2} give a unified way of describing many (previously unconnected) semifields, and thus also a unified way to study the corresponding semifield planes:
\begin{problem}
	Do the projective planes coordinatized by semifields constructed via Constructions~\ref{thm:general_structure} and~\ref{thm:general_structure_2} have any unifying characteristic properties?
\end{problem}

\bibliographystyle{acm}

\bibliography{semifields}
\end{document}